\documentclass[preprint,11pt]{elsarticle}

\usepackage{amssymb}
\usepackage{amsthm,amsmath}
\usepackage{comment}
\usepackage{xcolor}
\allowdisplaybreaks

\newtheorem{theorem}{Theorem}
\newtheorem{remark}{Remark}
\newtheorem{assumption}{Assumption}
\newtheorem{corollary}{Corollary}
\newtheorem{proposition}{Proposition}
\newtheorem{lemma}{Lemma}
\newdefinition{definition}{Definition}

\newcommand{\cP}{\mathcal{P}}

\newcommand{\jb}{\mathbf{j}}
\newcommand{\cU}{\mathcal{U}}
\newcommand{\cG}{\mathcal{G}}
\newcommand{\cT}{\mathcal{T}}
\newcommand{\InDeg}{\mathcal{N}}

\newcommand{\ib}{\mathbf{i}}
\newcommand{\ui}{\underline{\ib}}
\newcommand{\cD}{\mathcal{D}}
\newcommand{\ceil}[1]{\left\lceil #1 \right\rceil}

\usepackage{mathrsfs}
\usepackage{amsmath}

\usepackage{xcolor,cancel}
\usepackage[normalem]{ulem}

\journal{Stochastic Processes and their Applications}

\begin{document}

\begin{frontmatter}



\title{Local weak limits for collapsed branching processes with random out-degrees}


\author{Sayan Banerjee, Prabhanka Deka, and Mariana Olvera-Cravioto}

\affiliation{organization={Department of Statistics and Operations Research},
            addressline={University of North Carolina, Chapel Hill, CB \#3260}, 
            city={Chapel Hill},
            postcode={27599}, 
            state={NC},
            country={USA}}

\begin{abstract}
We obtain local weak limits in probability for Collapsed Branching Processes (CBP), which are directed random networks obtained by collapsing random-sized families of individuals in a general continuous-time branching process. The local weak limit of a given CBP, as the network grows, is shown to be a related continuous-time branching process stopped at an independent exponential time. The proof involves the construction of an explicit coupling of the in-components of vertices with the limiting object. We also show that the in-components of a finite collection of uniformly chosen vertices locally weakly converge (in probability) to i.i.d. copies of the above limit, reminiscent of propagation of chaos in interacting particle systems. We obtain as special cases novel descriptions of the local weak limits of directed preferential and uniform attachment models. We also outline some applications of our results for analyzing the limiting in-degree and PageRank distributions. In particular, upper and lower bounds on the tail of the in-degree distribution are obtained and a phase transition is detected in terms of the growth rate of the attachment function governing reproduction rates in the branching process.
\end{abstract}




\end{frontmatter}



\section{Introduction}

We analyze in this paper an evolving directed random graph model that is obtained by collapsing a continuous-time branching process driven by a general Markovian pure birth process. Our model corresponds to a graph process where incoming individuals (nodes) arrive in families, or groups, each having a random number of individuals. Upon arrival, each member of the family chooses one other existing node with probability proportional to a function $f$ (called the attachment function) of its degree (one plus the number of inbound edges), and connects to it using a directed outbound edge.  The members of each family arrive sequentially, so that if the first family has $D_1^+$ members, then the $(D_1^++1)$th individual belongs to the second family. For a graph with $n$ groups, the process continues according to this rule until $S_n = D_1^+ + \dots + D_n^+$ individuals have been connected to the graph, at which point we proceed to merge all the individuals in each family into a single ``group vertex". The result is a directed multigraph $G(V_n, E_n)$, having vertices $V_n = \{1, 2, \dots, n\}$ and edges in the set $E_n$, that models the connections among families. Our choice of notation $D_i^+$ for the size of the $i$th family comes from the observation that $D_i^+$ becomes the out-degree (number of outbound edges) of vertex $i$ in $V_n$ (with the exception of vertex 1, whose out-degree will be $D_1^+-1$). 

When the attachment function is linear (or constant), $G(V_n, E_n)$ constructed in this way corresponds to a directed preferential attachment graph (respectively, a directed uniform attachment graph) with random out-degrees and random additive fitness (see \eqref{upa}). In this case, vertices in the graph can be thought of as individuals, who upon arrival choose a random number of vertices to connect to, with edges always pointing from younger vertices to older ones.  

For the case where $D_i^+$ is a fixed constant $d$ for all $i$, the above model was introduced in \cite{garavaglia2018trees} under the name of Collapsed Branching Process (CBP), and we will use this nomenclature for our more general model. The limiting degree distribution was investigated in \cite{garavaglia2018trees}. The analysis of the model with random out-degrees and, in particular, the description of local weak limits (asymptotics of neighborhoods of typical vertices in large networks) were left as open problems.

The main focus of this paper is to describe the local weak limit for the collapsed branching process (CBP) graph for a general form of the attachment function. Local weak convergence is informally described as the phenomenon where finite neighborhoods of a uniformly chosen vertex in a growing sequence of finite graphs converges weakly to the neighborhood of the root in some limiting graph (finite or infinite). This concept was introduced in \cite{aldous2004objective,benjamini2011recurrence} and has turned out to be an indispensable tool in understanding the local geometry of large graphs through, for example, the degree and PageRank distributions (see Section \ref{S.Properties} for the definition and properties of this centrality measure popularized by Google), the size of giant components and the behavior of random walks on them. Moreover, local weak convergence has been used to show that a variety of random graphs are \emph{locally tree-like},  that is, non-tree graph sequences have local limits which are trees. This phenomenon also applies to our model, as will be seen below.
See \cite{van2023random} for a detailed treatment of local weak convergence and its applications. 

For random graphs, most of the existing results on local weak convergence concern \emph{static} random graphs like Erd\H os-Renyi graphs, inhomogeneous random graphs and the configuration model \cite[Chapters 2, 3, 4 and references therein]{van2023random}. Here, the randomness comes from the degree distribution and edge connection probabilities but there is no temporal evolution, making the roles of vertices \emph{exchangeable}. Consequently, the local weak limits for most of those graphs correspond to Galton-Watson trees where, in particular, every vertex in the same generation has the same progeny distribution. However, there has been comparatively little work on local weak convergence of dynamic random graphs (graphs evolving over time). The main obstacle is that the time evolution assigns ages to the vertices and the local geometry around a vertex depends crucially on its age. The local weak limit of the CBP in the tree case was obtained in \cite{rudas-2}, building on the work of \cite{nerman1981convergence,jagers-nerman-1}, where the limit was shown to be distributed as the same CBP but stopped at an independent exponential time. The directed preferential attachment model DPA($d,\beta$) with fixed out-degrees $d$ and fixed additive fitness $\beta$ (CBP with all out-degrees $d$ and attachment function $f(k) = k + \beta/d$) was analyzed in \cite{berger2014asymptotic} using an encoding of the graph in terms of a P\'olya urn type scheme (see \cite[Theorem 2.1]{berger2014asymptotic}) to describe the local limit as the so-called \emph{P\'olya-point graph} \cite[Section 2.3.2]{berger2014asymptotic}. The DPA($1,\beta$) with random $\beta$ was studied in \cite{lo2021weak}. Local limits for dynamic random trees where the attachment probabilities are non-local functionals of the vertex (like its PageRank) were recently obtained in \cite{banerjee2022co}. The local weak limit of preferential attachment type models with random i.i.d.~out-degrees and fixed additive fitness was derived in \cite{garavaglia2022universality} using a generalization of the P\'olya-point graph.  We remark here that the P\'olya representation of the pre-limit graph process, which forms the starting point of the results in \cite{berger2014asymptotic,lo2021weak,garavaglia2022universality}, is intrinsic to a linear attachment function and does not extend to more general attachment schemes.

Another related model was studied in \cite{diemor} where, for every new vertex $u$, an edge $(u,v)$ is created with probability proportional to f(in-degree of $v$)/[current network size] independently for each existing vertex $v$ in the network. Among other things, the authors in \cite{diemor} obtained a local limit (jointly for the in- and out-components) for this network. Although this model incorporates general attachment functions and the dynamics lead to random out-degrees of vertices, the attachment probabilities have a \emph{non-random denominator} (which is also true for DPA($d,\beta$)). This provides a crucial technical simplification that does not apply to the CBP case. Moreover, the out-degree distribution is rigidly tied to the network dynamics and one cannot construct this network with a prescribed out-degree distribution.

The main contribution of the current article is the construction of an explicit coupling between the exploration of the in-component of a uniformly chosen vertex in the CBP graph (with general $f$ satisfying mild assumptions) and its local weak limit, described by a marked continuous-time branching process (CTBP) stopped at an independent exponential time. The CTBP appearing in the limit falls into the broad category of Crump-Mode-Jagers (CMJ) branching processes (see \cite{nerman1981convergence,jagers-nerman-1,rudas-2}) where individuals reproduce independently according to a point process obtained by superimposing a random number (distributed as the out-degree) of independent copies of a Markovian pure birth process $\xi_f$ (see \eqref{suppp}). The rate of the exponential time at which the CTBP is stopped equals the \emph{Malthusian rate} of the CMJ process driven by $\xi_f$ (see Assumption \ref{A.MainAssum}(i) below). Our coupling is well-defined for the simultaneous exploration of the in-components of \emph{all the vertices} in the graph. In particular, the in-components of any finite collection of uniformly chosen vertices are successfully coupled with high probability with i.i.d. copies of the local weak limit. This is reminiscent of the phenomenon of \emph{propagation of chaos} in interacting particle systems \cite{chaintron2022propagation}. This type of coupling, which could include in applications the addition of vertex attributes that may depend on the out-degrees, has been established for static graphs in \cite{olvera2022strong}, where it is referred to as a {\em strong coupling}. Our results for the strong coupling of the CBP graph are summarized in Theorem~\ref{T.Main} and yield local weak convergence in probability as a corollary (Corollary~\ref{wconp}). Note that our results establish, as special cases, the local weak convergence of (the in-components in) the DPA($d,\beta$) graph and the CBP tree, previously obtained respectively in \cite{berger2014asymptotic} and \cite{rudas-2}. The local weak convergence implies the joint distributional convergence of the empirical in-degree and PageRank distributions  (Corollary \ref{degPR}). Asymptotics of the in-degree distribution for regularly varying out-degrees are quantified for the preferential and uniform attachment models in Proposition \ref{lincase}. 

In Proposition \ref{P.logasymptotics}, we obtain upper and lower bounds for the in-degree distribution for more general attachment functions in terms of functionals $G_k(\ell) := \sum_{j=1}^{\ell}\frac{1}{f(j)^k}, \, k,\ell \in \mathbb{N}$. These bounds are essentially sharp when $G_2(\infty) < \infty$. When $G_2(\infty)=\infty$, the bounds imply a \emph{strictly heavier tail} for the in-degree distribution, thereby exhibiting a \emph{phase transition}. A similar phenomenon was observed previously in the context of `persistence' in \cite{persist2021} and it signifies a transition from `strong correlation' to `weak correlation' between the age and in-degree of vertices in the growing network process. This phase transition is also captured in the fluctuation asymptotics of the point process $\xi_f$ driving the local limit, leading to a `central limit theorem' in Proposition \ref{L.WeakConv}.

There are several motivating factors behind this project. Firstly, as indicated by Propositions \ref{P.logasymptotics} and \ref{L.WeakConv}, the local limits for CBP `interpolate' between the uniform and linear attachment mechanisms, which exhibit very different qualitative and quantitative behavior, and capture precisely the sensitivity of the network geometry to the attachment mechanism (and associated phase transitions). Secondly, as the proofs of our results show, the developed technique does not rely on the explicit form of the attachment function (unlike \cite{berger2014asymptotic,lo2021weak,garavaglia2022universality}) and is therefore applicable to a wider variety of dynamic random networks. For example, consider the \emph{non-linear preferential attachment model} where every new vertex attaches to the network via \emph{more than one outgoing edge} (non-tree case), each independently choosing a parent vertex with probability proportional to a non-linear function of its current degree. Although this model does not fall under the CBP class, it admits a different (albeit more complicated) collapsing procedure (see \cite[Lemma 3.3]{banerjeedegcent}). We hope that the general approach laid out here will extend to such networks.

Thirdly, our description of the limit, beyond being applicable for a much wider class of attachment functions, is analytically better suited for obtaining refined asymptotics of a host of network functionals like degree and PageRank, in comparison to the P\'olya-point graph appearing in \cite{berger2014asymptotic,garavaglia2022universality}. This is because the local limit is a true (randomly stopped) CTBP where each vertex reproduces according to a superimposition of i.i.d.~Markovian pure birth processes, as opposed to Poisson processes with random intensities, that vary across vertices in an involved manner, which appear in the description of the P\'olya-point graph.
As a consequence, one can apply a host of tools from the well-studied theory of continuous-time Markov chains and CMJ processes (see the discussion after Corollary~\ref{degPR}). This was already exhibited in \cite{banerjee2022pagerank} where the P\'olya-point graph was redescribed as a randomly stopped CMJ process (our local limit for fixed out-degree and linear $f$) to compute the tail exponent of the limiting PageRank distribution in the DPA($d,\beta$) model and shed light on the \emph{power-law hypothesis} for dynamic random graphs. Similar descriptions of the local weak limit for other dynamic network models were recently exploited in \cite{banerjee2022co} to exhibit phase transitions in degree and PageRank behavior. Lastly, we point out that our only requirement for the out-degree distribution is a finite first moment compared to higher moments required by \cite{garavaglia2022universality}.

We note that, stemming from our interest in the asymptotic behavior of the degree and PageRank distributions, our couplings and local limits apply to the in-components only, rather than to the joint in- and out-components described in \cite{berger2014asymptotic,garavaglia2022universality}. However, as observed in the tree case in \cite{aldous1991asymptotic,rudas-2}, one can often obtain the joint local limit from that of the in-component using ideas from `fringe distributions' defined in \cite{aldous1991asymptotic}. More precisely, if the limiting in-component distribution satisfies certain `fringe' properties (expected root degree equal to one in the tree case \cite[Proposition 3]{aldous1991asymptotic}), then there is a unique infinite random tree, constructed in terms of `fringe trees' attached to an `infinite spine', which gives the joint local limit of the in- and out-components. We believe that a generalization of this approach will furnish the joint local limit for the CBP, and the limiting infinite tree can be described in terms of a forest of independent CMJ processes (run until different random times) emanating from vertices of the limiting out-component, which corresponds to a certain Galton-Watson tree. We leave this analysis for a later work.

In ongoing and future work, we exploit our description of the local limit for the CBP to obtain more refined asymptotics for centrality measures like degree, PageRank and betweenness centrality, and exhibit several interesting phase transitions. We are also working towards establishing a general technique for obtaining joint local limits of in- and out-components from that of the in-component for non-tree network models.

The rest of the paper is organized as follows. Section~\ref{S.Construction} provides a detailed construction of the CBP; Section~\ref{S.LocalLimit} contains our main theorem establishing the strong coupling, and Section~\ref{S.Properties} describes some applications of the local limit to the analysis of the PageRank and in-degree distributions. Finally, Section~\ref{S.Proofs} contains all the proofs, with the description of the coupling in subsections \ref{SS.Coupling} and \ref{SS.Coupling1}.

\section{The collapsed branching process} \label{S.Construction}

To construct a collapsed branching process (CBP) with random out-degrees we start by defining a continuous time branching process (CTBP) $\boldsymbol{\xi}$ driven by a Markovian pure birth process $\{ \xi_f(t) : t \geq 0\}$ satisfying $\xi_f(0) = 0$ and having birth rates
$$P( \xi_f(t+dt) = k+1 | \xi_f(t) = k) = f(k+1) dt + o(dt),$$
where $f: \mathbb{N} \to \mathbb{R}_+$ is a nonnegative function. Number each of the nodes in this CTBP according to the order of their arrival, with the root being labeled node 1. Let $\sigma_i$ denote the time of arrival of node $i$ in the CTBP and let $\mathcal{T}(t)$ denote the discrete skeleton of the graph determined by $\boldsymbol{\xi}$ at time $t$, with directed edges pointing from an offspring to its parent. 

Independently of the CTBP $\boldsymbol{\xi}$, we will also construct a sequence of i.i.d.~random variables $\{D_i^+: i \geq 1\}$ taking values on $\mathbb{N} :=\{1, 2, \dots \}$ and having distribution function $H(x) = P( D_1^+ \leq x)$ with $\mu := E[D_1^+] < \infty$. We will use this sequence $\mathbf{D}_n = \{ D_i^+: 1 \leq i \leq n\}$ in combination with $\boldsymbol{\xi}$ to construct a vertex-weighted directed graph $G(V_n, E_n)$ with vertex set $V_n := \{1, 2, \dots, n\}$.  
Let $S_k = D_1^+ + \dots + D_k^+$, $S_0 = 0$.

To start, define the sets
$$V(i) = \{ S_{i-1}+1, S_{i-1} + 2, \dots, S_i\}, \qquad i \geq 1.$$
The directed graph $G(V_n, E_n)$ is obtained by collapsing all the nodes in $V(i)$ into the vertex $i$, matching the outbound edges of its $D_i^+$ nodes with the merged vertices their parents in $\mathcal{T}(\sigma_{S_n})$ belong to. Note that $G(V_n, E_n)$ is a multigraph, since it may contain self-loops and multiple edges in the same direction between the same two vertices. Figure~\ref{F.Collapse} illustrates the collapsing procedure and the resulting multigraph.

\begin{figure}[t] 
\begin{center}
\includegraphics[scale=0.9, bb = 70 570 550 760, clip]{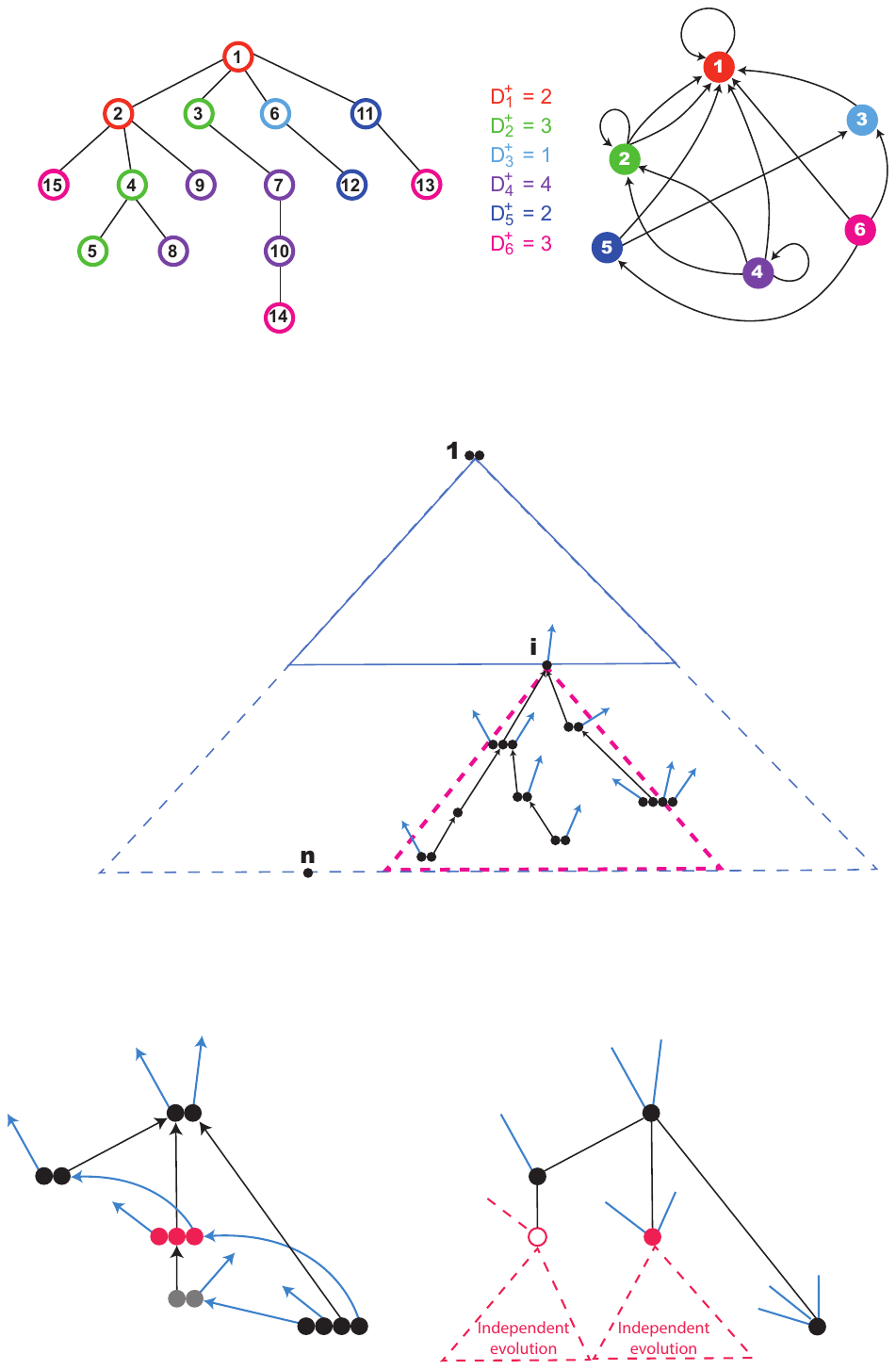}
\caption{Collapsed branching process. On the left the tree $\mathcal{T}(\sigma_{S_6})$, on the right the corresponding graph $G(V_6, E_6)$.} \label{F.Collapse}
\end{center}
\end{figure}

Note that if $f$ is linear, i.e., $f(k) = ck + \beta$ for some constants $c,\beta$ satisfying $c+\beta > 0$, then $G(V_n, E_n)$ can be seen as an evolving random directed rooted graph sequence $\{ G_\ell \}_{\ell \geq 1}$ where $G_1$ contains one vertex having $D_1^+-1$ self-loops and one extra half-edge pointing away from this vertex, and for $\ell \geq 2$, $G_{\ell}$ is constructed from $G_{\ell-1}$ by adding one vertex, labeled $\ell$, having $D_\ell^+$ outbound edges that are connected, one at a time, with the $k$th edge, $1 \leq k \leq D_{\ell}^+$, choosing to connect to vertex $i$ with probability:
\begin{align}\label{upa}
&P\left( \left. \text{$k^{th}$ outbound edge of $\ell$ attaches to $i$} \right| G_{\ell-1,k-1}, D_\ell^+ \right) \notag\\
&\qquad = 
\begin{cases}
\frac{cD_i(\ell-1,k-1)+ \beta D_{i}^+}{\sum_{j=1}^{\ell-1} (cD_j(\ell-1,k-1) + \beta D_j^+) + cD_\ell(\ell-1,k-1)+ \beta (k-1)},
 &   1 \le i \le \ell-1,\\\\
\frac{cD_\ell(\ell-1,k-1)+ \beta (k-1)}{\sum_{j=1}^{\ell-1} (cD_j(\ell-1,k-1) + \beta D_j^+) + cD_\ell(\ell-1,k-1)+ \beta (k-1)},
  & i = \ell,
\end{cases}
\end{align}
where $G_{\ell-1,k-1}$ is the graph after $k-1$ edges of vertex $\ell$ have been attached and $D_i(\ell-1,k-1)$ is the total degree of vertex $i$ (i.e. in-degree plus out-degree) in this graph. The case $c =1$ corresponds to a directed preferential attachment graph, while the case $c = 0$ corresponds to the directed uniform attachment graph, both with random out-degrees distributed according to $H$ and \emph{random additive fitness} (coming from the $\beta D^+_i$ term). Note that the name `uniform attachment graph' is a slight misnomer as, after the collapsing, the probability of a younger vertex connecting to an older one is not the same for every older vertex (rather it is proportional to the latter's out-degree). However, we stick to this nomenclature to emphasize that the branching process, before collapsing, indeed corresponds to a uniform attachment progeny graph.

\begin{remark}
To avoid confusion, we will always refer to the vertices in $G(V_n, E_n)$ as ``vertex/vertices", while we will refer to individuals in the CTBP $\boldsymbol{\xi}$, or its discrete skeleton $\{ \mathcal{T}(t): t \geq 0\}$, as ``node/nodes". 
\end{remark}

\section{Coupling with a marked continuous time branching process} \label{S.LocalLimit}

The existence of a local weak limit for $G(V_n, E_n)$ requires that we impose some conditions on the function $f$ that drives the CTBP $\boldsymbol{\xi}$.

To start, define 
\begin{align*}
\hat \rho(\theta) &:= E\left[ \int_0^\infty e^{-\theta s} \xi_f(ds) \right] = E\left[ \sum_{n=1}^\infty e^{-\theta \tau_n }  \right] = E\left[  \sum_{n=1}^\infty e^{-\theta \sum_{i=1}^{n} \chi_{i}/f(i)}  \right] \\
&= \sum_{n=1}^\infty \prod_{i=1}^n E\left[   e^{-\theta  \chi_{i}/ f(i)} \right]  =   \sum_{n=1}^\infty \prod_{i=1}^n \frac{1}{\theta/ f(i) + 1} = \sum_{n=1}^\infty \prod_{i=1}^{n} \frac{f(i)}{\theta + f(i)} ,
\end{align*}
where $\{ \chi_i: i \geq 1\}$ is a sequence of i.i.d.~exponential random variables with rate one, and $\tau_n$ is the time of the $n$th birth of $\{ \xi_f(t): t \geq 0\}$. 

\begin{assumption} \label{A.MainAssum}
Suppose the out-degrees $\{ D_i^+: i \geq 1\}$ are i.i.d., with $\mu = E[D_1^+] < \infty$. In addition, suppose the following hold:
\begin{enumerate}
\item[(i)] There exists $\lambda > 0$ such that $\hat \rho(\lambda) = 1$.
\item[(ii)] $f(k) \leq C_f k$, $k \geq 1$ for some constant $C_f < \infty$.
\item[(iii)] $f_* := \inf_{i \geq 1} f(i) > 0$.
\item[(iv)] Let $\underline \theta := \inf\{ \theta > 0: \hat \rho(\theta) < \infty\}$ and suppose that 
$$\lim_{\theta \searrow \underline{\theta}} \hat \rho(\theta) > 1.$$
\end{enumerate}
\end{assumption}

Note that $\lambda > 0$ in Assumption~\ref{A.MainAssum} is the Malthusian rate of $\{ \xi_f(t): t \geq 0\}$. 

The local weak limit of $G(V_n, E_n)$ is given by a marked continuous time branching process, whose discrete marked skeleton (the graph obtained by removing time labels from the nodes but retaining their marks) at time $t$ will be denoted $\mathcal{T}^c(t, \boldsymbol{\mathcal{D}})$, where $\boldsymbol{\mathcal{D}} :=\{\mathcal{D}_k : k \geq 1 \}$ is an i.i.d.~sequence having distribution $H$. The marks $\boldsymbol{\mathcal{D}}$ play a role in the construction of the local weak limit, and become vertex marks in its discrete skeleton. To describe this CTBP, define for each $k \ge 1$
\begin{equation}\label{suppp}
\bar{\xi}_f^{(k)} = \sum_{i=1}^{\mathcal{D}_k} \xi_f^{k,i},
\end{equation}
where $\mathcal{D}_k$ is the $k$th element of $\boldsymbol{\mathcal{D}}$ and the $\{ \xi_f^{k,i}: i\geq 1, k \geq 1\}$ are i.i.d.~copies of $\xi_f$. Each node in the tree is indexed in the order in which it arrives, and node $k$ has as its mark $\mathcal{D}_k$.
Let $\mathcal{T}^c(t,\boldsymbol{\mathcal{D}})$ denote the discrete marked skeleton at time $t$ of a marked CTBP driven by $\{ (\mathcal{D}_k, \bar \xi_f^{(k)}): k \geq 1\}$, conditionally on the root being born at time $t=0$.  We will denote the corresponding unmarked discrete skeleton by $\mathcal{T}^c(t)$.

Throughout the paper we will use $\mathcal{G}_i^{(n)}$ to denote the subgraph of $G(V_n, E_n)$ rooted at vertex $i$ that corresponds with the exploration of its in-component, where the exploration is such that we only follow the inbound edges, but also keep track of the out-degrees $\{D_j^+: j \in \mathcal{G}_i^{(n)} \}$ as we go; however, we do not follow the outbound edges. As before, $\mathcal{G}_i^{(n)}$ denotes the unmarked graph, while $\mathcal{G}_i^{(n)}(\mathbf{D}^+)$ denotes the marked one. With some abuse of notation, we will write $j \in \mathcal{G}_i^{(n)}$ to refer to a vertex in $\mathcal{G}_i^{(n)}$. 

\begin{definition} \label{D.MarkIsomorphic}
We say that two multigraphs $G(V, E)$ and $G(V', E')$ are {\em isomorphic} if there exists a bijection $\theta: V \to V'$ such that $l(i) = l(\theta(i))$ and $e(i,j) = e(\theta(i), \theta(j))$ for all $i \in V$ and all $(i,j) \in E$, where $l(i)$ is the number of self-loops of vertex $i$ and $e(i,j)$ is the number of edges from vertex $i$ to vertex $j$. In this case, we write $G \simeq G'$. 
\end{definition}

For nodes in trees, we use the symbol $\emptyset$ to denote the root and enumerate its vertices with labels of the form $\mathbf{i} = (i_1, \dots, i_k) \in \mathbb{N}^k$, where $(i_1, \dots, i_k, i_{k+1})$ is the $i_{k+1}$th offspring of node $(i_1, \dots, i_k)$; nodes with labels $i \in \mathbb{N}$ are offspring of the root $\emptyset$. Define the Ulam-Harris set $\mathcal{U} = \bigcup_{k=0}^\infty \mathbb{N}^k$, with the convention that $\mathbb{N}^0 \equiv \{ \emptyset \}$.  With a slight abuse of notation, we will write $\mathcal{D}_{\mathbf{j}}$ to refer to the mark of a node labeled $\mathbf{j} \in \mathcal{U}$. As mentioned earlier, we will also index nodes in dynamic trees in the order in which they arrive. In this case, for $j \in \mathbb{N}$, $j \in \mathcal{T}$ will denote the $j$th node to arrive in the tree $\mathcal{T}$.

Now we state the main result in the paper. Recall that $\lambda$ denotes the Malthusian rate of $\{\xi_f(t):t\geq0\}$ defined in Assumption \ref{A.MainAssum} (i).

\begin{theorem} \label{T.Main}
Suppose Assumption~\ref{A.MainAssum} holds. Then, for the CBP $G(V_n, E_n)$ we have:
\begin{enumerate} 
\item[i)] For $n \in \mathbb{N}$, if $I_n$ is uniformly chosen in $V_n$, independently of anything else, then, there exists a coupling $\cP_n$ of $\mathcal{G}_{I_n}^{(n)}(\mathbf{D}^+)$ and $(\chi, \{\mathcal{T}^c(t,\boldsymbol{\mathcal{D}}) : t \ge 0\})$, where $\boldsymbol{\mathcal{D}} :=\{\mathcal{D}_k : k \geq 1 \}$ is an i.i.d.~sequence having distribution $H$ and $\chi$ is an exponential random variable with rate $\lambda$, independent of $\{ \mathcal{T}^c(t,\boldsymbol{\mathcal{D}}): t \geq 0\}$, such that the event
$$C_{I_n} = \left\{  \mathcal{G}_{I_n}^{(n)} \simeq \mathcal{T}^c(\chi), \, \bigcap_{\mathbf{j} \in \mathcal{T}^c(\chi) } \{ D^+_{\theta(\mathbf{j})} = \mathcal{D}_{\mathbf{j}} \}  \right\},$$
where $\theta: \mathcal{U} \to \mathbb{N}$ is the bijection defining $\mathcal{G}_{I_n}^{(n)} \simeq \mathcal{T}^c(\chi)$, satisfies
$$\cP_n \left(C_{I_n} \right) \rightarrow 1, \quad \text{ as } \quad n \to \infty.$$

\item[ii)] Fix $m \in \mathbb{N}$. For $n \in \mathbb{N}$ and $\{ I_{n,j}: 1 \leq j \leq m\}$ i.i.d.~and uniformly chosen in $V_n$, independently of anything else, there exists a coupling $\cP_{n,m}$ of $\left(\mathcal{G}_{I_{n,j}}^{(n)} (\mathbf{D}^+) \right)_{1 \le j \le m}$ and i.i.d.~copies of $(\chi, \{\mathcal{T}^c(t,\boldsymbol{\mathcal{D}}) : t \ge 0\})$, denoted $(\chi_j, \{\mathcal{T}^c_j(t,\boldsymbol{\mathcal{D}}) : t \ge 0\})_{1 \le j \le m}$, such that the events $C_{I_{n,j}}$ defined as in part (i) satisfy
$$\cP_{n,m}\left( \bigcap_{j=1}^m C_{I_{n,j}} \right) \rightarrow 1, \quad \text{ as } \quad n \to \infty.$$

\end{enumerate}
\end{theorem}

Theorem~\ref{T.Main} implies, in particular, the local weak convergence in probability of $G(V_n, E_n)$, recorded in the following corollary.

\begin{corollary} \label{wconp}
Suppose Assumption~\ref{A.MainAssum} holds. For any fixed finite tree $T$ and any deterministic sequence $\{d_\jb : \jb \in \cU \} \subseteq \mathbb{N}$, 
$$\frac{1}{n} \sum_{i=1}^{n} 1 \left( \cG^{(n)}_{i} \simeq T, \,  \bigcap_{\jb \in T} \{ D^+_{\theta_i(\jb)} = d_{\jb} \}  \right) \overset{P}{\rightarrow} P \left( \cT^c(\chi) \simeq T, \, \bigcap_{\jb \in T} \{\mathcal{D}_{\jb} = d_{\jb} \} \right),$$
where $\theta_i: \mathcal{U} \to \mathbb{N}$ is the bijection defining $\mathcal{G}_i^{(n)} \simeq T$ and $\mathcal{T}^c(\chi,\boldsymbol{\mathcal{D}})$ is the marked discrete skeleton of the (randomly stopped) CTBP from the theorem.  
\end{corollary}
In the sequel, we will suppress the dependence of the marked discrete skeleton  $\mathcal{T}^c(\cdot,\boldsymbol{\mathcal{D}})$ on the marks $\boldsymbol{\mathcal{D}}$, and simply write $\mathcal{T}^c(\cdot)$, when there is no risk of ambiguity.

\section{Applications of the local limit} \label{S.Properties}

In this section we give some basic applications of the local limit $\mathcal{T}^c(\chi)$ for understanding asymptotic properties of the distributions of the  in-degree and the PageRank of a typical vertex in the original CBP. We will assume throughout that Assumption~\ref{A.MainAssum} holds.

PageRank, introduced by Brin and Page \cite{page1999pagerank}, is a celebrated centrality measure on networks whose goal is to 
rank vertices in a graph  according to their `popularity'.  Specifically, the PageRank score of vertex $v$ corresponds to the long-run proportion of time that a certain random walk spends on vertex $v$, hence the `popularity' interpretation. The PageRank  of a vertex is known to be heavily influenced by its in-degree and by the PageRank scores of its close inbound neighbors \cite{olvera2019pagerank}. Formally, PageRank is defined as follows. 

Let $G=G(V,E)$ be a directed network with vertices $V$ and edge set $E$. For each vertex $v \in V$, let $d_v^-$ and $d_v^+$ denote its in-degree and out-degree, respectively. Writing $|V|$ for the number of vertices in the graph and the vertices as $\{1,\dots,|V|\}$, let $A$ denote the adjacency matrix of $G$ defined as the $|V| \times |V|$ matrix whose $(i,j)$th element is the number of directed edges from vertex $i$ to vertex $j$. Let $\Delta$ be the diagonal matrix whose $i$th element is $1/d_i^+$ if $d_i^+>0$ and $0$ if $d_i^+=0$. Let $P$ be the matrix product $\Delta A$ with the zero rows replaced with the probability vector ${\bf q} = |V|^{-1} \mathbf{1}$. Note that $P$ is a stochastic matrix (all rows sum to one). The PageRank vector $\boldsymbol{\pi} = (\pi_1, \dots, \pi_{|V|})$, with \emph{damping factor} $c \in (0,1)$, is defined as the solution to the following system of equations:
$$\boldsymbol{\pi} = \boldsymbol{\pi} (cP) + (1-c) \mathbf{q}.$$
As the matrix $cP$ is substochastic (its rows sum to $c$), the system of equations is guaranteed to have a unique solution given by:
$$\boldsymbol{\pi} = (1-c)\mathbf{q}(I - cP)^{-1} = (1-c) \mathbf{q} \sum_{k=0}^\infty (c P)^k.$$
We will consider the \emph{scale-free} PageRank $\boldsymbol{R} := |V|\boldsymbol{\pi}$. For $n \in \mathbb{N}$ and $1 \le i \le n$, let $D^-_i(n),R_i(n)$ denote the in-degree and (scale-free) PageRank of the $i$th vertex in the CBP $G(V_n, E_n)$. Let $\InDeg_\emptyset, \mathcal{R}_\emptyset$ be the in-degree and PageRank of the root in the coupled marked tree $\mathcal{T}^c(\chi)$ from Theorem \ref{T.Main}. Then the following holds as a corollary of Theorem \ref{T.Main}.

\begin{corollary} \label{degPR}
For any continuity point $r$ of the distribution function of $\mathcal{R}_\emptyset$ and $k \in \mathbb{N} \cup \{0\}$, we have
\begin{equation*}
\frac{1}{n}\sum_{i=1}^n1(D^-_i(n) \ge k, R_i(n) > r) \xrightarrow{P} P(\InDeg_\emptyset \ge k, \mathcal{R}_\emptyset > r)
\end{equation*}
as $n \to \infty$.
\end{corollary}

Note that, while $\boldsymbol{R}$ is defined in a deterministic fashion for a given directed graph, the randomness of our underlying graphs makes $\boldsymbol{R}$ a random vector and $\mathcal{R}_{\emptyset}$ a non-negative random variable. 

The above corollary follows from Corollary \ref{wconp} exactly as in the proof of \cite[Theorem 4.6]{banerjee2022pagerank} (see also \cite[Theorem 2.1]{garavaglia2020local}) and its proof is omitted. Corollary~\ref{degPR} not only identifies the limiting joint in-degree and PageRank distribution, but the limiting objects are especially tractable from the point of view of moment computations and large deviations analysis, as they are defined as explicit functionals of the continuous time Markov chain $\{\mathcal{T}^c(t) : t \ge 0\}$ and an independent exponential random variable $\chi$. Specifically, 
$
\InDeg_\emptyset \stackrel{d}{=} \bar{\xi}_f^{(k)}(\chi) = \sum_{i=1}^{\mathcal{D}_k} \xi_f^{k,i}(\chi)$,
and $\mathcal{X}_\emptyset := \mathcal{R}_\emptyset/\mathcal{D}_\emptyset$ satisfies a renewal-type equation which we now describe. For $t \ge 0$, let $\InDeg_\emptyset(t)$ and $\mathcal{R}_\emptyset(t)$ denote the in-degree and PageRank of the root $\emptyset$ in $\mathcal{T}^c(t)$. Let $\{\sigma^\emptyset_i\}_{i \ge 1}$ denote the birth times of the children of the root in $\{\mathcal{T}^c(t) : t \ge 0\}$. For $t \ge 0$ and $i \ge 1$, let $\mathcal{D}_i$ and $\mathcal{R}_i(t)$ denote the mark (out-degree) and PageRank, respectively, of the $i$th child of $\emptyset$ in $\mathcal{T}^c(t +  \sigma^\emptyset_i)$.  Write $\mathcal{X}_\emptyset(t) := \mathcal{R}_\emptyset(t)/\mathcal{D}_\emptyset$ and $\mathcal{X}_i(t) := \mathcal{R}_i(t)/\mathcal{D}_i$. Then we have the following identity:
$$
\mathcal{X}_\emptyset \stackrel{d}{=} \mathcal{X}_\emptyset(\chi) = \frac{c}{\mathcal{D}_\emptyset}\sum_{i=1}^{\InDeg_\emptyset(\chi)}\mathcal{X}_i(\chi-\sigma^\emptyset_i)  + \frac{1-c}{\mathcal{D}_\emptyset}.
$$
The above identity has a nice recursive structure because, conditionally on $\{\sigma^\emptyset_i\}_{i \ge 1}$, $\{\mathcal{X}_i(\cdot) : i \ge 1\}$ are i.i.d. having the same distribution as $\mathcal{X}_\emptyset(\cdot)$. This identity enables the detailed analysis of the distribution of $\mathcal{X}_\emptyset$ through renewal theoretic techniques \cite{nummelin1978uniform,niemi1986non}, generator-based methods \cite{banerjee2022pagerank}, and the extensive theory of Crump-Mode-Jagers branching processes \cite{jagers-nerman-1,nerman1981convergence}. This approach was already used in  \cite{banerjee2022pagerank} for the DPA($d,\beta$) model.

A future companion paper to this work will present the large deviations behavior of the distributions of the in-degree and the PageRank of a typical vertex in the general CBP. As that work will show, this behavior is heavily determined by the attachment function $f$ and spans the entire range of distributions, from exponential tails to regularly varying ones. However, for illustration purposes, we include here some observations about the limiting in-degree distribution in the preferential and uniform attachment cases.

\begin{proposition} \label{lincase}
Let $\InDeg_\emptyset$ denote the in-degree of the root of $\mathcal{T}^c(\chi)$. Let $h(d) := H(d) - H(d-1) = P(\mathcal{D}=d), \, d \in \mathbb{N},$ be the pmf of the node marks. Then,
\begin{enumerate}
\item {\rm \bf Preferential attachment:} If $f(k) = k + \beta$, with $\beta > -1$, then for any $x \in \mathbb{N} \cup \{0\}$,
$$
P(\InDeg_\emptyset = x) = \sum_{d=1}^{\infty}h(d)(2 + \beta) \frac{\Gamma(2+\beta + d(\beta+1)) \, \Gamma(x + d(\beta + 1))}{\Gamma(d(\beta + 1)) \, \Gamma(x + d(\beta + 1) + 3 + \beta)},
$$
where $\Gamma(\cdot)$ is the Gamma function. In particular, if $h(d) = d^{-\gamma} L(d)$ for some finite $\gamma \geq 2$ and slowly varying function $L(\cdot)$ (with $E[\mathcal{D}]< \infty$), then 
\begin{align*}
P( \InDeg_\emptyset = x) &= (1+o(1))  K_\gamma P(XY > x) 
\end{align*}
as $x \to \infty$, where $P(X > x)= (1 + \beta + x)^{-3-\beta}$ for $x > -\beta$ is a Type II Pareto random variable, independent of $Y$, and $P(Y = d) = d^{-\gamma-1} l(d) L(d)/K_\gamma$ for $d \in \mathbb{N}$, where $K_\gamma = \sum_{d=1}^\infty d^{-\gamma-1} l(d) L(d)$ and
$$l(d) := \frac{(2+\beta) \Gamma(2+\beta + d(\beta+1)) }{d^{2+\beta} \Gamma(d(\beta + 1))} \to  (2+\beta) (\beta+1)^{2+\beta}, \qquad d \to \infty.$$
Moreover, $x \mapsto P(XY > x)$ is regularly varying with tail index $\min\{3+\beta,\gamma\}$, and
$$K_\gamma  P(XY > x) = (1+o(1)) \begin{cases} K_\gamma  E[Y^{3+\beta}] x^{-3-\beta}, & \text{if } \gamma > 3+\beta  \\
E[ (X^+)^\gamma] l(\infty) L(x) x^{-\gamma}, & \text{if } 2 \leq \gamma < 3+\beta, \end{cases}$$
as $x \to \infty$.

\item {\rm \bf Uniform attachment:}  If $f(k) \equiv \beta$ for some $\beta > 0$, then for any $x \in \mathbb{N} \cup \{0\}$,
$$
P(\InDeg_\emptyset = x) = \sum_{d=1}^{\infty}h(d)\frac{1}{d+1}\left(1 + \frac{1}{d}\right)^{-x}.
$$
In particular, if $h(d) = d^{-\gamma} L(d)$ for some finite $\gamma\geq 2$ and slowly varying function $L(\cdot)$ (with $E[\mathcal{D}]< \infty$), then 
\begin{align*}
P(\InDeg_\emptyset = x)  = (1+o(1)) E[ W^\gamma]   x^{-\gamma} L(x) 
 \end{align*}
as $x \to \infty$, where $W$ is an exponential random variable with mean one. 

\end{enumerate}
\end{proposition}
The above result shows in particular that, for the preferential attachment case with regularly varying out-degree distribution with exponent $\gamma$, the in-degree is regularly varying with the same exponent as the out-degree if $\gamma \le 3 + \beta$. Otherwise, the tail exponent matches the degree exponent in the tree case (our model with all out-degrees equal to one). Comparing this result to the degree distribution in the preferential attachment model with random out-degrees but deterministic additive fitness studied in \cite{deijfen2009preferential} (see \cite[Proposition~1.4]{deijfen2009preferential}), we observe that if we choose in our model the fitness parameter to be $\beta = \delta/\mu$ (to make it on average $\delta$ after the collapsing procedure) and set the additive fitness parameter of the model in \cite{deijfen2009preferential} to be $\delta$, then we obtain the same transition in the tail exponent for both models.

The model in \cite{deijfen2009preferential} can also be exactly obtained via a closely related CBP where the attachment function $f_v$ of an incoming vertex $v$ depends on its out-degree $D^+$ as $f_v(k) = k + \beta/D^+$, $k \in \mathbb{N}$. Since the construction of the CBP can be done conditionally on the out-degree sequence, one can obtain local limits for this variant using a similar construction to the one used here. 

For the uniform attachment case, we observe the (somewhat surprising) phenomenon that, although the in-degree distribution has exponential tails for deterministic out-degrees, making the out-degree distribution regularly varying also makes the in-degree distribution regularly varying with the same exponent.

The explicit description of the local weak also leads to upper and lower bounds for the tail of the degree distribution for general attachment functions satisfying Assumption~\ref{A.MainAssum}. 

Define the functions 
\begin{align}
G_k(\ell) = \sum_{j=1}^{\ell}\frac{1}{f(j)^k}, \qquad k,\ell \in \mathbb{N}. \label{E.Gi}
\end{align}
The functions $G_i$ are extended to $\mathbb{R}_+$ by linear interpolation (setting $G_i(0)=0$).
Then we have the following.
\begin{proposition}\label{P.logasymptotics}
    Suppose Assumption~\ref{A.MainAssum} holds. Let $\mathcal{N}_{\emptyset}$ denote the in-degree of the root of $\mathcal{T}^c(\chi).$ Let $\cD$ denote a random variable with out-degree distribution $H$. Then we have that for any $x>0$,
    \begin{align*}
        P(\mathcal{N}_{\emptyset} > x ) &\leq  E \left[ (2\cD + 1)e^{-\lambda G_1(\lceil \frac{x}{\cD} \rceil) + 2\lambda^2 G_2(\lceil \frac{x}{\cD} \rceil)}\right],\\
        P(\mathcal{N}_{\emptyset} > x ) &\geq \frac{1}{2}E \left[e^{-\lambda (G_1( \lceil \frac{x}{\mathcal{D}}\rceil )  + \frac{4f(1)^2}{f_*}\log (2\mathcal{D}) \, G_2(\lceil \frac{x}{\mathcal{D}}\rceil))}\right].
    \end{align*}
\end{proposition}
\begin{remark}
For fixed out-degree $m$, Proposition \ref{P.logasymptotics} shows that the tail exponent of the in-degree distribution is of leading order $-\lambda G_1(x/m)$. Moreover, when $G_2(\infty) < \infty$, this quantity gives the precise tail asymptotics of the in-degree distribution in the sense $C_1 e^{-\lambda G_1(x/m)} \le P(\mathcal{N}_{\emptyset} > x ) \le C_2 e^{-\lambda G_1(x/m)}, \, x>0,$ for positive finite constants $C_1,C_2$. If $G_2(\infty) = \infty$, the tail is strictly heavier than $e^{-\lambda G_1(x/m)}$ (owing to the now non-negligible $G_2$ term), thereby exhibiting a \textbf{phase transition}. This phase tansition has been previously observed in the context of `persistence' in \cite{persist2021} where the age and in-degree of vertices in the network exhibit strong (positive) correlation when $G_2(\infty) < \infty$ and weak correlation when $G_2(\infty) = \infty$.  
More careful computations will give the exact tail asymptotics for general sublinear attachment functions, similar to the results for affine attachment functions in Proposition~\ref{lincase}, and this is work in progress.
\end{remark}

The phase transition noted above is also exhibited by the fluctuation asymptotics of the point process $\xi_f$ driving the CTBP in the local weak limit. 
\begin{proposition}\label{L.WeakConv}
Suppose $f(x) = x^\alpha L(x) > 0$ with $\alpha \in [0, 1]$; if $\alpha = 1$ suppose that $\lim_{x \to \infty} L(x) = 0$. Let 
$$V(t) = f(r(t)) G_2(r(t))^{1/2}, \qquad r(t) = \lceil G_1^{-1}(t) \rceil.$$
Then, there exists a random variable $Z$, such that
$$\frac{\xi_f(t) - G_1^{-1}(t)}{V(t)} \Rightarrow Z, \qquad t \to \infty.$$
Moreover, $Z$ has a standard normal distribution when $G_2(\infty) = \infty$ and has moment generating function
$$\kappa(\theta) = E\left[ e^{\theta Z} \right] =  e^{\sum_{k=2}^\infty \frac{(-\theta)^{k}}{k} a_k}, \qquad a_k = G_k(\infty)/G_2(\infty)^{k/2}$$
when $G_2(\infty) < \infty$. 
\end{proposition}

Other related works studying the effect on the degree distribution of heavy-tailed fitness parameters include \cite{lodewijks2020phase} and  \cite{lodewijks2020maximaldegree}. In \cite{lodewijks2020phase}, the authors analyze directed preferential attachment networks with constant out-degree and random additive fitness. They also consider a related model with random out-degrees resulting from every new vertex connecting to each existing vertex (via at most one edge) with probability proportional to its in-degree plus a random additive fitness. In \cite{lodewijks2020maximaldegree}, the same authors consider a directed uniform attachment model with constant out-degree and random fitness, thought of as a random recursive tree with random vertex weights. Both these papers quantify the combined effect of the heavy-tailed fitness distribution and the network attachment mechanism on the limiting degree distribution and the maximal degree behavior. In comparison, the additive fitness appearing in the local limit of our CBP model, for the preferential and uniform attachment cases, is a scalar multiple of the out-degree. Nevertheless, the results in both \cite{lodewijks2020phase} and  \cite{lodewijks2020maximaldegree} are similar in spirit to Proposition~\ref{lincase} above, suggesting a more universal picture for dynamic network models with random additive fitness.

\section{Proofs} \label{S.Proofs}

The rest of the paper contains the proofs of Theorem~\ref{T.Main}, Corollary~\ref{wconp}, Proposition~\ref{lincase},  \ref{P.logasymptotics}, and \ref{L.WeakConv}. 

\subsection{Coupling Construction for $m=1$} \label{SS.Coupling}

Let $\{ \mathcal{T}(t): t \geq 0\}$ denote the discrete skeleton of the CTBP $\boldsymbol{\xi}$ described earlier. Recall that the graph $G(V_n, E_n)$ is obtained by simply collapsing the nodes in the sets $\{ V(i): 1 \leq i \leq n\}$. Clearly, it suffices to run $\boldsymbol{\xi}$ until the time $\sigma_{S_n}$.

Fix $i \in V_n$ and define 
$$t_{n,i} := \frac{1}{\lambda} \log(n/i).$$
We will now explain how to construct a coupling between $\mathcal{G}_i^{(n)}$ and a tree $\mathcal{T}^c_i(t_{n,i})$ that evolves according to the law of $\{\mathcal{T}^c(t): t \geq 0\}$. Recall that in order to simplify the notation, we have omitted the dependence on the marks $\mathbf{D}^+$ and $\boldsymbol{\mathcal{D}}$, and simply write  $\mathcal{G}_i^{(n)}$ and  $\mathcal{T}^c_i(t_{n,i})$ for the marked graphs. 

This coupling will only be successful with high probability for large values of $i$, so although it will be well-defined for any $i \geq 1$, it will most likely fail to satisfy $\mathcal{G}_i^{(n)} \simeq \mathcal{T}^c_i(t_{n,i})$ if $i$ is not sufficiently large. 

We start by sampling the out-degrees $\{D_i^+: 1 \leq i \leq n\}$ and the tree $\mathcal{T}(\sigma_{S_n})$. Note that at this point the entire graph $G(V_n, E_n)$ has been sampled, so steps 1-4 in the construction below are deterministic. Specifically, we will copy (re-use) some of the vertices in $\mathcal{G}_i^{(n)}$ and their birth times to construct $\mathcal{T}_i^c(t_{n,i})$, but potentially ignore others. To start, for $j \ge 1$, let $\kappa_i(j)$ denote the label in $G(V_n, E_n)$ of the $j$th oldest vertex in $\mathcal{G}_i^{(n)}$, with $\kappa_i(1) = i$. Nodes in $\{\mathcal{T}_i^c(t): t \geq 0\}$ will be of two kinds, those that we will copy from $\mathcal{G}_i^{(n)}$ and those that will be generated independently. The construction below collects the vertices that are copied from $\mathcal{G}_i^{(n)}$ onto nodes in $\{ \mathcal{T}_i^c(t): t \geq 0\}$ in the sets $J_i$ and $J_i^*$. Intuitively, $J_i$ tracks the labels of ``good" vertices that were copied over without breaking the coupling, while $J_i^*$ tracks the labels of ``bad" vertices that would result in a miscoupling. It will also create additional `dummy nodes' in $\{\mathcal{T}^c_i(t): t \geq 0\}$ when certain types of miscouplings occur, and those will be collected in the set $J_i^d$ but will receive no labels.

We will use $\tau_{i,j}$ to denote the birth time in $\mathcal{T}_i^c(\cdot)$ of the node corresponding to the $j$th vertex to be added to $J_i \cup J_i^*$. Vertices in $\mathcal{G}_i^{(n)}$ will be explored in the order of their ages. When the $j$th explored vertex has multiple nodes with outgoing edges to vertices in $J_i$, a dummy vertex is created for every such node except the oldest of these nodes. For each such node, say $\omega'$, the associated birth time in $\mathcal{T}^c_i$ is denoted by $\tau_{i,\omega'}$. The purpose of these dummy nodes is to ensure that the law of $\mathcal{T}_i^c(\cdot)$ agrees with that of $\mathcal{T}^c(\cdot)$ (even when the coupling between $\mathcal{T}_i^c(\cdot)$ and $\mathcal{G}_i^{(n)}$ is broken). The time in $\left\{ \mathcal{T}_i^{c}(t): t \geq 0\right\}$ will be tracked by the internal clock $s_i^*$. In the following construction, the last explored vertex in $J_i \cup J_i^*$ before the current vertex will be denoted by $\kappa^*$, which we will call the exploration number.

\begin{enumerate}
\item[1.] Fix $i \in V_n$ and initialize the internal clock $s_i^* = 0$ and the exploration number $\kappa^*= \kappa_i(1) = i$.  The root node of $\{\mathcal{T}^c_i(t): t \geq 0\}$ is born at time $\tau_{i,1} = 0$, i.e., $| \mathcal{T}_i^c(0)| = 1$, and is assigned as its mark $\mathcal{D}_{i,1} = D_i^+$.  

\item[2.] If none of the nodes in $V(i)$ create a self-loop, initialize the sets $J_i = \{i\}$ and $J_i^* = \varnothing$, and move on to the next step. Else, initialize the sets $J_i = \varnothing$ and $J_i^* = \{i \}$, and go to step 4. 


\item[3.] For $j = 2, \dots, |\mathcal{G}_i^{(n)}|$ do the following:
\begin{enumerate}
\item[a.] Determine $\kappa_i(j)$ on the exploration of $\mathcal{G}_i^{(n)}$. 

\item[b.] If there is a node in $V(\kappa_i(j))$ that attaches to a (node corresponding to a) vertex in $J_i$, go to step 3(c). Otherwise, update $j = j+1$ and go back to step 3. 

\item[c.] Set $\tau_{i,j} = s_i^* +  \sigma_{\omega}- \sigma_{S_{\kappa^*}}$, where $\omega$ is the  oldest node in $V(\kappa_i(j))$ connecting to a vertex in the set $J_i$. Let $\kappa_i(p)$ be the ancestor of $\omega$ in the set $J_i$. Update $s_i^* = \tau_{i,j}$ and add a node labelled $j$ to $\{ \mathcal{T}_i^c(t): t \geq 0\}$ born at time $\tau_{i,j}$, connected to $p$, and having mark $\mathcal{D}_{i,j} = D_{\kappa_i(j)}^+$.

\vspace{.5mm}

\item[d.] If $V(\kappa_i(j))$ creates no self loops nor multiple edges with vertices in $J_i$, update $J_i = J_i \cup \{ \kappa_i(j)\}$. Else, update $J_i^* = J_i^* \cup \{ \kappa_i(j)\}$. 
If $\kappa_i(j)$ was added to $J_i^*$ and there are $l \ge 2$ edges from nodes in $V(\kappa_i(j))$ to vertices in $J_i$, for each of the $l-1$ nodes $\omega'$ in $V(\kappa_i(j))$ with $\sigma_{\omega'} > \sigma_{\omega}$, create a `dummy node' and add it to $J_i^d$. Attach this node (without label) to $\mathcal{T}^c_i(\cdot)$, at the node corresponding to the vertex in $J_i$ where the associated edge coming from $\omega'$ was incident, and assign to this node the birth time $\tau_{i,\omega'} = s_i^* + \sigma_{\omega'} - \sigma_{\omega}$.

\item[e.] Set $j = j+1$, update $\kappa^* = \kappa_i(j)$, and go back to step 3. 

\end{enumerate}
\item[4.] Update the internal clock to $s_i^* =  s_i^* + \sigma_{S_n} - \sigma_{S_{\kappa^*}}$.

\item[5.] To each $\kappa_i(j) \in J_i^*$, attach an independent copy of $\mathcal{T}^c(s_i^* - \tau_{i,j})$ conditioned on its root having mark $\mathcal{D}_{i,j}$ (these are the vertices where miscouplings occurred). For a dummy node $\omega'$ in $J_i^d$, sample an independent mark $\mathcal{D}_{\omega'}$ from the out-degree distribution $H$ and attach an independent copy of $\mathcal{T}^c((s_i^* - \tau_{i,\omega'})^+)$ conditioned on the root having mark $\mathcal{D}_{\omega'}$.

\end{enumerate}

The construction returns $ \mathcal{T}_i^c(s_i^*)$. If $s_i^* < t_{n,i}$, let $\mathcal{T}_i^c(s_i^*)$ continue evolving according to the law of $\{ \mathcal{T}^c(t): t \geq 0\}$ until time $t_{n,i}$ (see Figures \ref{F.LiftedTree} and \ref{F.Exploration}). 

\begin{figure}[h] 
\begin{center}
\includegraphics[scale=0.8, bb = 70 300 550 540, clip]{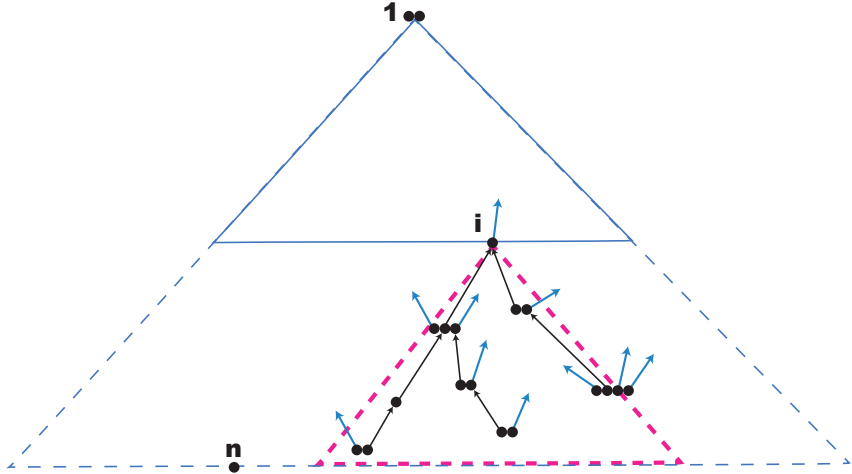}
\caption{The lifted tree $\mathcal{T}(\sigma_{S_n})$ and the in-component of vertex $i$, $\mathcal{G}_i^{(n)}$. Nodes in each family $V(j)$, $j \geq 1$ are depicted as if they were all born at the same time and are labeled according to the vertices of $G(V_n, E_n)$ they give rise to. In this figure, $V(1) = \{1,2\}$, $V(i) = \{S_i\}$ and $V(n) = \{S_n\}$. This depiction of $\mathcal{G}_i^{(n)}$ shows a successful coupling with its local limit.} \label{F.LiftedTree}
\end{center}
\end{figure}

\begin{remark}
Note that in steps 3(c) and 3(d) we allow nodes in $V(\kappa_i(j))$ to have outgoing edges to vertices in $J_i^*$ since these edges do not appear in the coupled tree $\mathcal{T}_i^c(t_{n,i})$ (descendants of nodes corresponding to vertices in $J_i^*$ will be generated independently in step 5). 
The tree structure is always preserved; the unique edge connecting node $j$ in $\mathcal{T}_i^c(t_{n,i})$ to its parent node (corresponding to a vertex in $J_i$) is copied from $\mathcal{G}_i^{(n)}$ in step 3(c). 
\end{remark}

\begin{figure}[h] 
\begin{center}
\begin{picture}(500,180)
\put(0,0){\includegraphics[scale=0.9, bb = 40 50 550 230, clip]{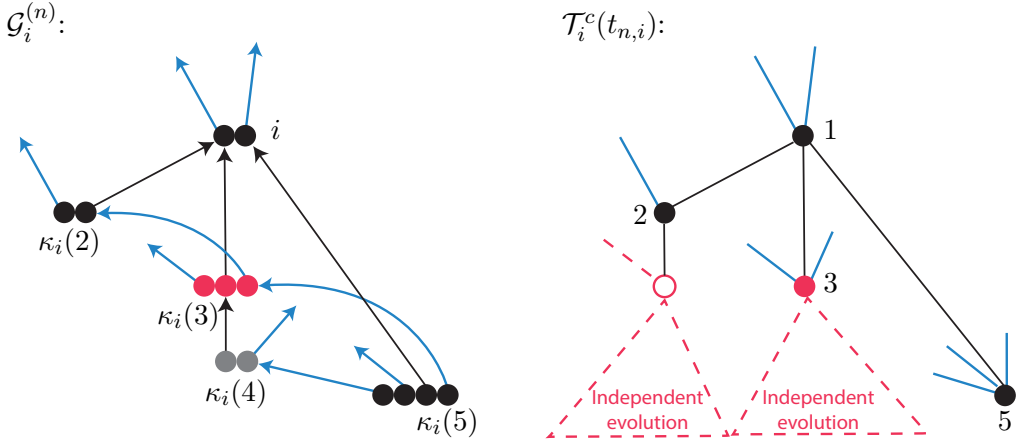}}
\put(10,160){$\mathcal{G}_i^{(n)}$:} 
\put(220,160){$\mathcal{T}_i^c(t_{n,i})$:} 
\put(110,120){$i$}
\put(319,120){$1$}
\put(67,50){$\kappa_i(3)$}
\put(319,62){$3$}
\put(22,78){$\kappa_i(2)$}
\put(247,88){$2$}
\put(85,21){$\kappa_i(4)$}
\put(165,10){$\kappa_i(5)$}
\put(385,10){$5$}
\end{picture}
\caption{Coupling of $\mathcal{G}_i^{(n)}$ and $\mathcal{T}_i^c(t_{n,i})$. On the left we have a depiction of $\mathcal{G}_i^{(n)}$, and on the right, its coupled tree $\mathcal{T}_i^c(t_{n,i})$. Vertices $i, \kappa_i(2)$, and $\kappa_i(5)$ are in $J_i$, while vertex $\kappa_i(3)$ is in $J_i^*$. Vertex $\kappa_i(4)$ is not in $J_i \cup J_i^*$ since it was skipped in step 3(b). The miscoupling caused by vertex $\kappa_i(3)$ generated one `dummy node' in $\mathcal{T}_i^c(t_{n,i})$ with an independently generated mark, the one depicted as an offspring of node $2$. Vertex $\kappa_i(5)$ did not cause a miscoupling since it did not create a self-loop nor did it attach to more than one vertex in $J_i$.} \label{F.Exploration}
\end{center}
\end{figure}

\begin{remark} \label{R.Miscoupling}
Note that the coupling between $\mathcal{G}_i^{(n)}$ and $\mathcal{T}_i^c(t_{n,i})$ can break for one of the following two reasons:
\begin{enumerate}
\item[a.] $|J_i^*| \geq 0$, which means miscouplings in steps 2 or 3(d) occurred, or,
\item[b.] $|J_i^*| = 0$ but $|\mathcal{T}_i^c(t_{n,i})| \neq |\mathcal{G}_i^{(n)}|$, which would happen if $\{\mathcal{T}_i^c(t): t \geq 0\}$ has births in between times  $s_i^*$ and $t_{n,i}$. 
\end{enumerate}
\end{remark}

\subsection{Coupling Construction for $m \geq 2$} \label{SS.Coupling1}

For some finite $m \geq 2$ and $\underline{\ib}= \{i_1, \ldots, i_m \} \subseteq [n]$, we similarly construct a coupling between $\{\cG_{i_1}^{(n)}, \ldots, \cG_{i_m}^{(n)} \}$ and independent trees $\{\cT_{i_1;\underline{\ib}}^{c}(t_{n,i_1}), \ldots, \cT_{i_m;\underline{\ib}}^{c}(t_{n,i_m}) \}$, all of which evolve according to the law of $\{\cT^c(t) : t \geq 0 \}$.

As for the $m=1$ case, sample the out-degrees $\{ D_i^+ : 1 \leq i \leq n \}$ and the tree $\cT(\sigma_{S_n})$. Without loss of generality, assume $i_1 \leq i_2 \leq \ldots \leq i_m$. Set $\{ \cT_{i_1;\ui}^{c}(t) : t \geq 0\} = \{ \cT_{i_1}^c(t) : t \geq 0 \}$ following the steps for the $m=1$ case. Further, define the sets $J_{i_1;\ui} = J_{i_1}$, $J_{i_1;\ui}^* = J_{i_1}^*$, $S_{\ui} = \{ \kappa_{i_1}(j) : 1 \leq j \leq |G_{i_1}^{(n)}| \}$, and the internal clock $s_{i_1;\ui}^* = s_{i_1}^*$.\\
To construct $\cT_{i_2;\ui}^{c}(t_{n,i_2})$, we again use the same construction, with the only difference being that the set $J_{i_2}^*$ is replaced by a larger set $J_{i_2;\ui}^*$ : a vertex is put in $J_{i_2;\ui}^*$ if it creates a loop or multiple edges with vertices in (the current) $J_{i_2}$ or any vertex in $S_{\ui}$. Vertices in $J_{i_2;\ui}^*$ and the dummy nodes in $J_{i_2;\ui}^d$ undergo independent subsequent evolution as described in Step 5 of the coupling construction. We also require that for vertices in $J_{i_2}^* \subseteq J_{i_2;\ui}^*$, we use the same independent copies in the construction of Step 5. Return the internal clock time $s_{i_2;\ui}$ and the tree $\cT_{i_2;\ui}^{c}(s_{i_2;\ui})$, and update the set $S_{\ui} = S_{\ui} \cup \{ \kappa_{i_2}(j) : 1 \leq j \leq |G_{i_2}^{(n)}| \}$.\\
Iterate the above process successively for $i_3, \ldots, i_m$ to construct $m$ trees $$\{\cT_{i_1;\ui}^{c}(s_{i_1;\ui}), \ \ldots, \cT_{i_m;\ui}^{c}(s_{i_m; \ \ui})\},$$ which are all independent of each other and their evolution has the same law as $\{\cT^c(t) : t \geq 0 \}$.

\begin{remark} \label{R.Miscoupling1}
Note that the coupling between $(\mathcal{G}_{i_1}^{(n)}, \ldots, \cG_{i_m}^{(n)})$ and \linebreak $(\cT_{i_1;\ui}^{c}(t_{n,i_1}), \ldots, \cT_{i_m;\ui}^{c}(t_{n,i_m}) )$ can break for one of the following two reasons:
\begin{enumerate}
\item[a.] $|J_{i_{\ell};\ui}^*| > 0$ for some $1 \leq \ell \leq m$, which means miscouplings in steps 2 or 3(d) occurred, or,
\item[b.] $|J_{i_{\ell};\ui}^*| = 0$ for all $1 \le \ell\le m$ but $|\mathcal{T}_{i_{\ell};\ui}^c(t_{n,i_{\ell}})| \neq |\mathcal{G}_{i_{\ell}}^{(n)}|$ for some $\ell$, which would happen if $\{\mathcal{T}_{i_{\ell};\ui}^c(t): t \geq 0\}$ has births in between times  $s_{i_{\ell};\ui}^*$ and $t_{n,i}$. 
\end{enumerate}
\end{remark}

\subsection{Proof of Theorem~\ref{T.Main}}

The proof of Theorem~\ref{T.Main} is split into several lemmas. Although the description of the coupling is given by first sampling $\{ D_i^+: 1 \leq i \leq n\}$ and then constructing the tree $\{ \mathcal{T}(t): t \geq 0\}$ until time $\sigma_{S_n}$, note that the two can be done simultaneously. This will be the approach followed through most of the proofs in this section, under the assumption that $D_i^+$ is sampled from distribution $H$ at time $\sigma_{S_{i-1}}$. 

To ease the reading of this section we have compiled the notation that is used repeatedly: 
\begin{itemize}
\item $\mathscr{G}_t$ denotes the filtration generated by the construction of $\{\mathcal{T}(t): t \geq 0\}$ along with the sequence $\{ D_i^+: i \geq 1\}$ up to time $t$. 
\item $\mu = E[D_1^+]$ denotes the expected out-degree.
\item $\sigma_k$ denotes the birth time of the $k$th node in $\{ \mathcal{T}(t): t \geq 0\}$.
\item $S_k = D_1^+ + \dots + D_k^+$.
\item $V(i) = \{ S_{i-1}+1, S_{i-1}+2, \dots, S_i\}$.
\item $\{ \mathcal{T}^c(t): t \geq 0\}$ denotes the generic version of the discrete skeleton of a CTBP driven by $\{ (\mathcal{D}_k, \bar \xi_f^{(k)}): k \geq 1\}$.
\item  $\Lambda(t) = |\mathcal{T}^c(t)| + \sum_{j=1}^{|\mathcal{T}^c(t)|} \mathcal{D}_{j}$ denotes the number of nodes plus the sum of the marks in $\mathcal{T}^c(t)$. 
\item $\mathcal{G}_i^{(n)}$ denotes the subgraph of $G(V_n, E_n)$ obtained from exploring the in-component of vertex $i$, with the out-degrees of its vertices as marks.
\item $\kappa_i(j)$ denotes the label in $G(V_n, E_n)$ of the $j$th oldest vertex in $\mathcal{G}_i^{(n)}$.
\item $\{ \mathcal{T}_i^c(t): t \geq 0\}$ denotes the discrete skeleton of the tree rooted at vertex $i \in G(V_n, E_n)$ at time $t$, as constructed in the coupling from Section~\ref{SS.Coupling}. 
\item $\{ \mathcal{D}_{i,j}: j \geq 1\}$ denotes the marks of the nodes in $\{ \mathcal{T}_i^c(t): t \geq 0\}$. 
\item $D^{-}_v(t)$ denotes the in-degree of node $v$ in $\mathcal{T}(t)$.
\item $\mathscr{F}^{(i)}_{t}$ denotes the sigma-algebra generated by $\{\mathcal{T}_i^c(s), \{ \mathcal{D}_{i,j}: j \in \mathcal{T}_i^c(s)\} : 0 \leq s \leq t \}$.
\item $\Lambda_i(t) = |\mathcal{T}_i^c(t)| + \sum_{j=1}^{|\mathcal{T}_i^c(t)|} \mathcal{D}_{i,j}$ denotes the number of nodes plus the sum of the marks in $\mathcal{T}^c_i(t)$. 
\item $s_i^*$ denotes the time at which the construction of $\{ \mathcal{T}_i^c(t): t \geq 0\}$ ends in Section~\ref{SS.Coupling}. 
\item $J_i$ denotes the set of vertices of $\mathcal{G}_i^{(n)}$ that were successfully coupled to nodes in $\mathcal{T}_i^c(s_i^*)$.
\item $J_i^*$ denotes the set of vertices of $\mathcal{G}_i^{(n)}$ that caused miscouplings with $\mathcal{T}_i^c(s_i^*)$.
\item $s_{i_{\ell},\ui}, J_{i_{\ell};\ui} ; 1 \leq \ell \leq m$  denotes the corresponding objects in the construction in Section~\ref{SS.Coupling1}.
\end{itemize}


\begin{lemma} \label{L.TreeSize}
For any $t > 0$, we have
$$E\left[ | \mathcal{T}^c(t) | + \sum_{j \in \mathcal{T}^c(t)} \mathcal{D}_j \right] \leq 1 + \mu e^{C_f(\mu + 1)t}.$$
\end{lemma}

\begin{proof}
Let $\{ \xi_l(t): t \geq 0\}$ be a Markovian pure birth process satisfying $\xi_l(0) = 0$ and having birth rates
$$P(\xi_l(t+dt) = k+1 | \xi_l(t) = k) = C_f (k+1) dt + o(dt), \ k \ge 0.$$
Let $\{ \mathcal{D}_i: i \geq 1\}$ be an i.i.d.~sequence distributed according to $H$, and define 
$$\bar \xi_l^{(k)} = \sum_{i=1}^{\mathcal{D}_k} \xi_l^{k,i},$$
where the $\{ \xi_l^{k,i}: i \geq 1, k \geq 1\}$ are i.i.d.~copies of $\xi_l$. Let $\hat{\mathcal{T}}^c(t)$ be the discrete skeleton of a marked CTBP driven by $\{\bar \xi_l^{(k)}: k \geq 1\}$ at time $t \geq 0$ conditionally on the root being born at time $t = 0$. 

Then, by Assumption~\ref{A.MainAssum} we have that 
$$| \mathcal{T}^c(t) | + \sum_{j \in \mathcal{T}^c(t)} \mathcal{D}_j \leq_\text{s.t.} | \hat{\mathcal{T}}^c(t) | + \sum_{j \in \hat{\mathcal{T}}^c(t)} \mathcal{D}_j.$$

Let $X(t) := |\hat{\mathcal{T}}^c(t) | + \sum_{j \in \hat{\mathcal{T}}^c(t)} \mathcal{D}_j, t \ge 0.$ Note that, if $N_v(t)$ denotes the number of offspring of node $v$ in $\hat{\mathcal{T}}^c(t)$ and $h(m) = P(\mathcal{D}_1 = m)$, then for $t \ge 0$ and small $\Delta>0$,
\begin{align*}
&P(X(t+\Delta) = X(t) +m+1 \vert X(t)) \\
&= C_f\left(\sum_{v \in \hat{\mathcal{T}}^c(t)} (N_v(t) + \mathcal{D}_v)\right)h(m)\Delta + h(m)o(\Delta)\\
&= C_f(X(t) - 1)h(m)\Delta + h(m)o(\Delta).
\end{align*}
Hence, recalling $\mu = E[\mathcal{D}_1]$,
\begin{align*}
E[X(t+\Delta)] - E[X(t)] &= C_f\Delta \, E[X(t) -1]\sum_{m=1}^{\infty}(m+1)h(m) + o(\Delta)\\
& = C_f\Delta \, E[X(t) -1](\mu + 1) + o(\Delta).
\end{align*}
From this, writing $M(t) := E[X(t)]$, we obtain the differential equation
$$
M'(t) = C_f(\mu + 1)(M(t) - 1), \ t \ge 0.
$$
Solving this equation gives
$$
M(t) = 1 + (M(0) - 1)e^{C_f(\mu + 1)t}, \ t \ge 0.
$$
The lemma follows upon noting $M(0) = E[X(0)] = 1 + \mu$.
\end{proof}

\begin{lemma} \label{L.BirthTimes}
We have that
$$ \sup_{j,k \geq m} \left| \sigma_{k} - \sigma_j - \frac{1}{\lambda} \log(k/j) \right| \rightarrow  0, \qquad P\text{-a.s. as } m \to \infty.$$
\end{lemma}

\begin{proof}
Note that under Assumption~\ref{A.MainAssum}, using \cite[Theorem 6.3]{nerman1981convergence}, there exists an almost surely positive random variable $\Theta$ such that:
$$m e^{-\lambda \sigma_m} \rightarrow \Theta, \qquad P\text{-a.s. as } m \to \infty.$$

Equivalently, we have that:
$$- \sigma_m + \frac{1}{\lambda} \log m \rightarrow \frac{1}{\lambda} \log \Theta , \qquad P\text{-a.s. as } m \to \infty.$$
From here it follows that
\begin{align*}
 \sup_{j,k \geq m} \left| \sigma_{k} - \sigma_j - \frac{1}{\lambda} \log(k/j) \right| \leq 2  \sup_{k \geq m} \left| -\sigma_{k}  + \frac{1}{\lambda} \log k  - \frac{1}{\lambda} \log \Theta  \right| \to 0 \qquad P\text{-a.s.}
\end{align*}
as $m \to \infty$. 
\end{proof}

\begin{lemma} \label{L.TimeAdjustment}
We have 
\begin{enumerate} 
\item[i)] $$\lim_{n \to \infty} \frac{1}{n} \sum_{i=1}^n P\left( |\mathcal{T}_i^c(t_{n,i})| \neq |\mathcal{T}_i^c(\sigma_{S_n} - \sigma_{S_i})| \right) =0,$$
\item[ii)] $$\lim_{n \to \infty} \frac{1}{n^m} \sum_{\ui \subseteq [n] : |\ui| =m}  \sum_{\ell = 1}^{m} P \left(\left\{ |\cT_{i_{\ell};\ui}^c(t_{n,i_{\ell}})| \neq |\cT_{i_{\ell};\ui}^c(\sigma_{S_n} - \sigma_{S_\ell}) | \right\} \right) = 0.$$
\end{enumerate}
\end{lemma}

\begin{proof}
Fix $\epsilon > 0$ and note that
\begin{align}\label{int}
&P\left( |\mathcal{T}_i^c(t_{n,i})| \neq |\mathcal{T}_i^c(\sigma_{S_n} - \sigma_{S_i})| \right) \nonumber\\
&\leq P\left( |\mathcal{T}_i^c(t_{n,i}-\epsilon)| < |\mathcal{T}_i^c(t_{n,i}+\epsilon)| \right)  + P\left( \left| \sigma_{S_n} - \sigma_{S_i} - t_{n,i} \right| > \epsilon \right).
\end{align}

To bound the first probability let $\mathscr{F}^{(i)}_{t} = \sigma( \mathcal{T}_i^c(s), \{ \mathcal{D}_{i,j}: j \in \mathcal{T}_i^c(s)\} : 0 \leq s \leq t )$ and $\Lambda_i(t) = |\mathcal{T}_i^c(t)| + \sum_{j=1}^{|\mathcal{T}_i^c(t)|} \mathcal{D}_{i,j}$. Next, 
note that conditionally on $\mathscr{F}_{t_{n,i} -\epsilon}^{(i)}$, the next birth in $\{\mathcal{T}_i^c(t): t \geq 0\}$ will happen in an exponential time that has a rate that, by Assumption~\ref{A.MainAssum}, is bounded from above by
$C_f \Lambda_i(t_{n,i}-\epsilon).$
Therefore,
\begin{align}
&P\left( |\mathcal{T}_i^c(t_{n,i}-\epsilon)| < |\mathcal{T}_i^c(t_{n,i}+\epsilon)| \right) \notag   \\
&\leq E \left[  P\left(  |\mathcal{T}_i^c(t_{n,i} -\epsilon)| < |\mathcal{T}_i^c(t_{n,i} +\epsilon) |  \left| \mathscr{F}_{t_{n,i}-\epsilon}^{(i)} \right. \right) \right] \notag \\
&\leq  E \left[ P\left( \left. \text{Exp}(C_f \Lambda_i(t_{n,i}-\epsilon)) \leq 2\epsilon \, \right| \Lambda_i(t_{n,i}-\epsilon)  \right) \right]  \notag \\
&\leq E \left[ 1 - e^{-C_f 2\epsilon \Lambda_i(t_{n,i})}   \right]  . \label{eq:ExpBound}
\end{align}
Now let $I_n = \lceil nU \rceil$, where $U$ is a uniform $[0,1]$ independent of everything else, and note that by Lemma~\ref{L.BirthTimes} and the strong law of large numbers,
\begin{align*}
\lim_{n \to \infty} \left| \sigma_{S_n} - \sigma_{S_{I_n}} - t_{n, I_n} \right| &\leq \lim_{n \to \infty} \left| \sigma_{S_n} - \sigma_{S_{I_n}} - \frac{1}{\lambda} \log(S_n/S_{I_n}) \right| \\
&\hspace{5mm} + \lim_{n \to \infty} \frac{1}{\lambda} \left| \log\left( S_n I_n/(n S_{I_n}) \right) \right| = 0. \qquad P\text{-a.s.}
\end{align*}
Finally, letting $\chi = -(1/\lambda) \log U$, note that $t_{n,I_n} \leq \chi$, and use \eqref{int} to obtain that
\begin{align*}
&\frac{1}{n} \sum_{i=1}^n P\left( |\mathcal{T}_i^c(t_{n,i})| \neq |\mathcal{T}_i^c(\sigma_{S_n} - \sigma_{S_i})| \right) \\
&\leq   E \left[ 1 - e^{-C_f 2\epsilon \Lambda_{I_n}(t_{n,I_n})}   \right] +  P\left(  \left| \sigma_{S_n} - \sigma_{S_{I_n}} - t_{n, I_n} \right| > \epsilon \right) \\
&\leq E \left[ 1 - e^{-C_f 2\epsilon \Lambda(\chi)}   \right] + P\left(  \left| \sigma_{S_n} - \sigma_{S_{I_n}} - t_{n, I_n} \right| > \epsilon \right) \\
&\to  E \left[ 1 - e^{-C_f 2\epsilon \Lambda(\chi)}   \right],
\end{align*}
as $n \to \infty$. Now take $\epsilon \downarrow 0$ to complete the proof. \\
Part (ii) follows exactly as (i) upon using the union bound and the fact that, for any $t>0$, $\ui \subseteq [n]$ with $|\ui| =m$ and $\ell \in \{1,\dots,m\}$, $\cT_{i_{\ell};\ui}^c(t)$ has the same law as $\cT_{i_{\ell}}^c(t)$.
\end{proof}

Recall that $s_i^*$ denotes the internal clock at the end of the coupling construction in Section~\ref{SS.Coupling}.

\begin{lemma} \label{L.LargeTree}
For any $i \in V_n$ set $a_{i} = i^{1/2 -\delta} $ for some $0 < \delta < 1/2$, and define $\vartheta = C_f(\mu+1)/\lambda$. Let $\Lambda_i(t) = |\mathcal{T}_i^c(t)| + \sum_{j=1}^{|\mathcal{T}_i^c(t)|}\mathcal{D}_{i,j}$ and define the event
$$E_{n,i} = \left\{ \Lambda_i(t_{n,i} \vee s_i^*) \leq a_{i} \right\}.$$
Then, for any constant $c_\vartheta \in (\vartheta/(\vartheta+1/2-\delta),1)$, 
$$P(E_{n,i}^c) \leq (\mu+1) n^{-c_\vartheta(1/2+\vartheta-\delta)+ \vartheta} + 1( i < n^{c_\vartheta}) + P\left( |\mathcal{T}_i^c(t_{n,i})| \neq |\mathcal{T}_i^c(\sigma_{S_n} - \sigma_{S_i})| \right).$$
\end{lemma}

\begin{proof}
Note that since $s_i^* \leq \sigma_{S_n} - \sigma_{S_i}$, then for any $\epsilon > 0$ we have
\begin{align*}
&P(E_{n,i}^c) \\
&\leq P\left( |\mathcal{T}_i^c(t_{n,i})| + \sum_{j=1}^{|\mathcal{T}_i^c(t_{n,i})|} \mathcal{D}_{i,j} > a_i \right) + P\left( |\mathcal{T}_i^c(t_{n,i})| < |\mathcal{T}_i^c(\sigma_{S_n} - \sigma_{S_i})| \right) \\
&\leq P\left(\Lambda_i(t_{n,i}) > a_i\right) \left( 1(i \geq n^{c_\vartheta}) + 1(i < n^{c_\vartheta}) \right) + P\left( |\mathcal{T}_i^c(t_{n,i})| \neq |\mathcal{T}_i^c(\sigma_{S_n} - \sigma_{S_i})| \right)  \\
&\leq \frac{E[ \Lambda_i(t_{n,i})] }{a_i} 1(i \geq n^{c_\vartheta}) + 1(i < n^{c_\vartheta})  + P\left( |\mathcal{T}_i^c(t_{n,i})| \neq |\mathcal{T}_i^c(\sigma_{S_n} - \sigma_{S_i})| \right). 
\end{align*}

Now use Lemma~\ref{L.TreeSize} to obtain that for $\vartheta = C_f (\mu+1)/\lambda$, and any $1> c_\vartheta > \vartheta/(1/2+\vartheta - \delta)$, 
\begin{align*}
\frac{E[ \Lambda_i(t_{n,i})] }{a_{i}} 1(i \geq n^{c_\vartheta}) &\leq \frac{1 + \mu e^{C_f (\mu+1) t_{n,i}}}{a_i} 1(i \geq n^{c_\vartheta})\\
&\leq \frac{1 + \mu (n/i)^{\vartheta} }{i^{1/2-\delta} } 1(i \geq n^{c_\vartheta}) \\
&\leq (\mu+1) \frac{n^\vartheta}{i^{1/2+\vartheta-\delta}} 1(i \geq n^{c_\vartheta}) \\
&\leq (\mu+1) n^{-c_\vartheta(1/2+\vartheta-\delta)+ \vartheta}.
\end{align*}
\end{proof}

\begin{remark} \label{L.LargeTreeR}
The above lemma readily extends to $E_{n,i_{\ell};\ui} = \{\Lambda_{i_{\ell};\ui}(t_{n,i} \vee s_{i_{\ell};\ui}^*) \leq a_{i_\ell} \}$, where $\Lambda_{i_{\ell};\ui} = |\cT_{i_{\ell};\ui}(t)| + \sum_{j=1}^{|\cT_{i_{\ell};\ui}(t)|} \cD_{(i_{\ell};\ui),j}$ when using the construction for $m \geq 2$ in Section \ref{SS.Coupling1}. We obtain the same bound (replacing $i$ by $i_\ell$) for $P(E_{n,i_{\ell};\ui}^c)$.
\end{remark}

\begin{lemma} \label{L.Miscoupling}
\begin{enumerate} 
\item[i)] 
For any $i \in V_n$ set $a_{i} = i^{1/2 -\delta} $ for some $0 < \delta < 1/2$. Define the event $E_{n,i}$ as in Lemma~\ref{L.LargeTree}. 
Then, 
$$P\left( E_{n,i} \cap \left\{ |J_i^*| \geq 1 \right\} \right) \leq 1(i = 1) + \frac{2C_f \mu i^{-2\delta}}{f_*},$$
where $f_*$ is as defined in Assumption~\ref{A.MainAssum}.
\item[ii)] For any $\ui = (i_1, \ldots, i_m) \subseteq [n]$, define the events $E_{n,i_{\ell};\ui}$ as in Remark \ref{L.LargeTreeR} with $a_{i_\ell} = i_{\ell}^{1/2 - \delta}$ for some $0 < \delta < 1/2$. Then 
\begin{align*}
P\left(\bigcap_{\ell = 1}^{m} E_{n,i_{\ell};\ui} \cap \bigcup_{\ell=1}^{m} \{|J_{i_{\ell};\ui}^*| \geq 1\}  \right)& \leq 1\left(\min_{1 \leq \ell \leq m} i_{\ell} = 1 \right) \\
&\hspace{5mm} + \left(1 \wedge \frac{2C_fm^2\mu (\max_{1 \le \ell \le m} i_{\ell})^{1-2\delta}}{f_* (\min_{1 \leq \ell \leq m} i_{\ell})}\right).
\end{align*}
\end{enumerate}

\end{lemma}

\begin{proof}
i) Define the events
$$H_{i,j}  := \{\kappa_i(j) \text{ is the oldest vertex in } J_i^*\},$$
with $H_{i,j} = \emptyset$ if $j > |\mathcal{G}^{(n)}_i|$. Note that
\begin{align*}
P\left( E_{n,i} \cap \left\{ |J_i^*| \geq 1 \right\} \right) &=P\left( E_{n,i} \cap \bigcup_{j=1}^{a_i} H_{i,j} \right) = \sum_{j=1}^{a_i} P\left( E_{n,i} \cap H_{i,j} \right).
\end{align*}
To analyze the last probabilities, we focus on the construction of the tree $\{ \mathcal{T}^c_i(t): t\geq 0\}$, and note that for $H_{i,j}$ to happen it must be that one of the $D_{\kappa_i(j)}^+ = \mathcal{D}_{i,j}$ nodes in $V(\kappa_i(j))$ will attach to one of the younger nodes in $V(\kappa_i(j))$ or to one of the nodes in $\bigcup_{r=1}^{j-1} V(\kappa_i(r))$, and up until this point there have been no miscouplings. For time $t \ge 0$, denote by $s^*_i(t)$ the internal time in the tree process $\mathcal{T}^c_i(\cdot)$ accrued after all vertices in $\mathcal{G}^{(n)}_i$ with at least one node born by time $t$ have been explored. Define the event
$$E_i(t) = \left\{ |\mathcal{T}_i^c(s^*_i(t))| + \sum_{j =1}^{|\mathcal{T}_i^c(s^*_i(t))|} \mathcal{D}_{i,j} \leq a_i \right\},$$
which satisfies $E_{n,i} \subseteq E_i(t)$ for all $0 \leq t \leq t_{n,i} \vee s_i^*$. Next, let $S_{\kappa_i(j)-1} \leq \omega \leq S_{\kappa_i(j)}$ be the oldest node in $V(\kappa_i(j))$ that is born to $\bigcup_{r=1}^{j-1} V(\kappa_i(r))$, and note that conditionally on $\mathscr{G}_{\sigma_\omega}$ the event $H_{i,j}$ will happen if any of the nodes $v \in V(\kappa_i(j)), v > \omega$, is such that $v$ is born to one of the nodes in $U_v := \bigcup_{r=1}^{j-1} V(\kappa_i(r)) \cup \{ S_{\kappa_i(j)-1}, \dots, v-1\}$.
Let $\mathcal{H}_{i,j,v}$ be the event that node $v> \omega$ is the first node such that $v$ is born to a node in $U_v$. Then, recalling that $D^{-}_u(t)$ is the in-degree of node $u$ in $\mathcal{T}(t)$, using Assumption~\ref{A.MainAssum}, 
\begin{align*}
&1(E_i(\sigma_\omega)) P(H_{i,j} | \mathscr{G}_{\sigma_\omega}) \\
&\leq \sum_{v = \omega+1}^{S_{\kappa_i(j)}} 1(E_i(\sigma_\omega)) P( \mathcal{H}_{i,j,v} | \mathscr{G}_{\sigma_\omega}) \\
&= \sum_{v = \omega+1}^{S_{\kappa_i(j)}} E\left[ \left. 1(E_i(\sigma_{v-1})) P( \mathcal{H}_{i,j,v} | \mathscr{G}_{\sigma_{v-1}})  \right| \mathscr{G}_{\sigma_\omega} \right] \\
&=  \sum_{v = \omega+1}^{S_{\kappa_i(j)}} E\left[ \left. 1(E_i(\sigma_{v-1})) \cdot  \frac{\sum_{u\in U_v} f(D^{-}_u(\sigma_{v-1})+1)}{\sum_{u \in \mathcal{T}(\sigma_{v-1})} f(D^{-}_u(\sigma_{v-1})+1)}   \right| \mathscr{G}_{\sigma_\omega} \right] \\
&\leq  \sum_{v = \omega+1}^{S_{\kappa_i(j)}} E\left[ \left. 1(E_i(\sigma_{v-1})) \cdot  \frac{\sum_{u\in U_v} C_f(D^{-}_u(\sigma_{v-1})+1)}{\sum_{u \in \mathcal{T}(\sigma_{v-1})} f_*(v-1)}   \right| \mathscr{G}_{\sigma_\omega} \right]  \\
&\leq \frac{C_f}{f_* \omega} \sum_{v = \omega+1}^{S_{\kappa_i(j)}} E\left[ \left. 1(E_i(\sigma_{v-1})) \left(|\mathcal{T}_i^c(\sigma_{v-1})| + \sum_{k =1}^{|\mathcal{T}_i^c(\sigma_{v-1})|} \mathcal{D}_{i,k} \right)    \right| \mathscr{G}_{\sigma_\omega} \right] \\
&\leq \frac{C_f a_i \mathcal{D}_{i,j} }{f_* S_{\kappa_i(j)-1} } .
\end{align*}
It follows that
$$P(E_{n,i} \cap H_{i,j}) \leq \frac{C_f a_i }{f_*} E\left[ \frac{\mathcal{D}_{i,j}}{S_{\kappa_i(j)-1}} \right] ,$$
and since $S_i \geq i$ for all $i \geq 1$, 
\begin{align*}
P(E_{n,i} \cap \{ |J_i^*| \geq 1\}) &\leq  1 \wedge \left( \frac{C_f a_i }{f_*}  \sum_{j=1}^{a_i}E\left[ \frac{\mathcal{D}_{i,j}}{S_{\kappa_i(j)-1}} \right] \right) \leq  1\wedge \left( \frac{C_f \mu a_i^2  }{f_* (i-1)} \right) \\
& \leq  1(i = 1) + \frac{2C_f \mu i^{-2\delta}}{f_*} .
\end{align*}
To prove (ii), for any $\ui = (i_1, \ldots, i_m) \subseteq [n]$, define events $$H_{(i_{\ell};\ui),j} := \{\kappa_{i_{\ell}}(j) {\rm \ is \ the \ oldest \ vertex \ in \ } J_{i_{\ell};\ui}^* \},$$ with $H_{(i_{\ell};\ui),j} = \emptyset$ if $j > |G_{i_{\ell}}^{(n)}|$. Then we can write 
\begin{align*}
     &P\left(\bigcap_{\ell = 1}^{m} E_{n,i_{\ell};\ui} \cap  \bigcup_{\ell=1}^{m} \{|J_{i_{\ell};\ui}^*| \geq 1\} \right)\\
     &\quad  = P\left(\bigcap_{u = 1}^{m} E_{n,i_{u};\ui} \cap \left( \bigcup_{\ell =1}^{m} \left(\bigcap_{k=1}^{\ell-1} \{|J_{i_k;\ui}^*| = 0 \} \right) \cap \left(\bigcup_{j=1}^{a_{i_{\ell}}}H_{(i_{\ell};\ui),j} \right) \right) \right) \\
     &\quad\le \sum_{\ell=1}^{m} \sum_{j=1}^{a_{i_{\ell}}} P\left(\bigcap_{u = 1}^{m} E_{n,i_{u};\ui} \cap \ \left(\bigcap_{k=1}^{\ell-1} \{|J_{i_k;\ui}^*| = 0 \} \right) \cap H_{(i_{\ell};\ui),j} \right),
\end{align*}
where we use the convention that $\cap_{k=1}^0$ is the null set.
Note that for $H_{(i_{\ell};\ui),j}$ to happen, one of the $D_{\kappa_{i_\ell}(j)}^+ = \mathcal{D}_{i_\ell,j}$ nodes in $V(\kappa_{i_{\ell}}(j))$ will attach to one of the younger nodes in $V(\kappa_{i_{\ell}}(j))$ or to one of the nodes in $\bigcup_{k=1}^{\ell-1} \bigcup_{r=1}^{a_{i_k}} V(\kappa_{i_k}(r)) \cup \bigcup_{r=1}^{j-1}V(\kappa_{i_\ell}(r))$. Then, using the same arguments as in proof of part (i), we get that
 \begin{align*}
    P\left(\bigcap_{u = 1}^{m} E_{n,i_{u};\ui} \cap \ \left(\bigcap_{k=1}^{\ell-1} \{|J_{i_k;\ui}^*| = 0 \} \right) \cap H_{(i_{\ell};\ui),j} \right) &\leq \frac{C_f (\sum_{k=1}^{\ell} a_{i_k})}{f_* }E\left[ \frac{\mathcal{D}_{i_\ell,j}}{S_{\kappa_{\ell}(j)-1}}\right].
    \end{align*}
    Therefore,
    \begin{align*}
    &P\left(\bigcap_{\ell = 1}^{m} E_{n,i_{\ell};\ui} \cap  \bigcup_{\ell=1}^{m} \{|J_{i_{\ell};\ui}^*| \geq 1\} \right) \\
    &\leq 1 \wedge \left( \sum_{\ell=1}^{m} \frac{C_f (\sum_{k=1}^{\ell} a_{i_k})}{f_*}\sum_{j=1}^{a_{i_{\ell}}} E\left[ \frac{\mathcal{D}_{i_\ell,j}}{S_{\kappa_{i_\ell}(j)-1}}\right] \right)\\
    &\leq 1 \wedge \left(\frac{C_f m^2 \mu (\max_{1 \leq \ell \leq m} a_{i_{\ell}})^2}{f_* ( (\min_{1 \leq \ell \leq m} i_{\ell}) - 1)} \right) \\
    &\leq 1\left(\min_{1 \leq \ell \leq m} i_{\ell} = 1 \right) + \left(1 \wedge\frac{2C_fm^2\mu (\max_{1 \le \ell \le m} i_{\ell})^{1-2\delta}}{f_* (\min_{1 \leq \ell \leq m} i_{\ell})}\right).
\end{align*}
\end{proof}

\begin{remark}\label{unizero}
Let $U_1, \ldots, U_m$ be independent random variables, uniformly distributed in $[0,1]$, and set $I_{n,\ell} = \ceil{nU_{\ell}}$. Then it is routine to check that $$1 \wedge\frac{2C_fm^2\mu (\max_{1 \le \ell \le m} I_{n,\ell})^{1-2\delta}}{f_* (\min_{1 \leq \ell \leq m} I_{n,\ell})} \xrightarrow{P} 0$$ as $n \rightarrow \infty$. This fact, in conjunction with part (ii) of the above lemma, will be used in the proof of Theorem~\ref{T.Main}(ii).
\end{remark}

\begin{lemma} \label{L.Mistiming}
Let $E_{n,i}$ and $E_{n,i_{\ell};\ui}$ be the events defined in Lemma~\ref{L.LargeTree} and Remark \ref{L.LargeTreeR}. Then,
\begin{enumerate}
\item[i)] 
\begin{align*}
\lim_{n \to \infty} \frac{1}{n} \sum_{i=1}^n P\left( E_{n,i} \cap \left\{ |J_i^*| = 0 \right\} \cap \left\{  |\mathcal{T}_i^c(t_{n,i})| \neq |\mathcal{G}_i^{(n)}| \right\} \right) = 0.
\end{align*}
\item[ii)] \begin{align*} \lim_{n \to \infty} & \frac{1}{n^m} \sum_{\ui \subseteq [n] : |\ui|=m} P \left( \bigcap_{\ell=1}^{m} E_{n,i_{\ell};\ui} \cap  \bigcap_{\ell=1}^{m} \{|J_{i_{\ell};\ui}^*| = 0 \} \cap \bigcup_{\ell=1}^{m} \{ |\cT_{i_{\ell};\ui}^{c}(t_{n,i_{\ell}})| \neq |\cG_{i_{\ell}}^{(n)}| \}\right) = 0. 
\end{align*}
\end{enumerate}
\end{lemma}

\begin{proof}
Note that on the event $\{ |J_i^*| = 0\}$ we have that
$$\sigma_{S_n} - \sigma_{S_i}  - \sum_{j=2}^{|J_i|} (\sigma_{S_{\kappa_i(j)}} - \sigma_{S_{\kappa_i(j)-1}})  \leq s_i^* \leq \sigma_{S_n} - \sigma_{S_i} .$$
Moreover, since on the event $\{ |J_i^*| = 0\}$ we have that $\mathcal{T}_i^c(s_i^*) \simeq \mathcal{G}_i^{(n)}$,  in order for $ |\mathcal{T}_i^c(t_{n,i})| \neq |\mathcal{G}_i^{(n)}| $ to happen it must be that either $s_i^* < t_{n,i}$ and $\{ \mathcal{T}_i^c(t): t \geq 0\}$ had births in $(s_i^*, t_{n,i})$, or $s_i^* > t_{n,i}$ and  $\{ \mathcal{T}_i^c(t): t \geq 0\}$ had births in $(t_{n,i}, s_i^*)$. 

Fix $\epsilon > 0$, and note that 
\begin{align}
& P\left( E_{n,i} \cap \left\{ |J_i^*| = 0 \right\} \cap \left\{  |\mathcal{T}_i^c(t_{n,i})| \neq |\mathcal{G}_i^{(n)}| \right\} \right)  \notag \\ 
&= P\left( E_{n,i} \cap \left\{ |J_i^*| = 0 \right\} \cap \left\{  |\mathcal{T}_i^c(s_i^*)| < |\mathcal{T}_i^c(t_{n,i}) | \right\} \right)  \notag \\
&\hspace{5mm} +  P\left( E_{n,i} \cap \left\{ |J_i^*| = 0 \right\} \cap \left\{  |\mathcal{T}_i^c(t_{n,i})| < |\mathcal{T}_i^c(s_i^*) | \right\} \right)  \notag \\
&\leq P\left( E_{n,i} \cap \left\{ |J_i^*| = 0 \right\} \cap \left\{    |\mathcal{T}_i^c(t_{n,i} - \epsilon)| < |\mathcal{T}_i^c(t_{n,i})|  \right\} \right) \label{eq:Mistiming}  \\
&\hspace{5mm} + P\left( E_{n,i} \cap \left\{ |J_i^*| = 0 \right\} \cap \{ s_i^* < t_{n,i} - \epsilon \}  \right) \label{eq:Adjustment} \\
&\hspace{5mm} + P\left( |\mathcal{T}_i^c(t_{n,i})|  <  |\mathcal{T}_i^c(\sigma_{S_n} - \sigma_{S_i} )| \right)  . \notag
\end{align} 

Using the same steps leading to \eqref{eq:ExpBound}, we obtain that \eqref{eq:Mistiming} is bounded by
$$P\left(  |\mathcal{T}_i^c(t_{n,i} - \epsilon)| < |\mathcal{T}_i^c(t_{n,i})| \right) \leq E\left[ 1 - e^{- C_f \epsilon \Lambda_i(t_{n,i})} \right].$$

To analyze the probability in \eqref{eq:Adjustment}, note that by Assumption~\ref{A.MainAssum} we have that on the event $\{|J_i^*| = 0, s_i^* < t_{n,i} \}$,
$$\sum_{j=2}^{|J_i|} (\sigma_{S_{\kappa_i(j)}} - \sigma_{S_{\kappa_i(j)-1}})  \leq_\text{s.t.} \text{Erlang}\left( \sum_{j=1}^{|\mathcal{T}_i^c(t_{n,i})|} \mathcal{D}_{i,j}, \, S_i \right) =: \mathcal{E}_{i}. $$
Here, $\mathcal{E}_i$ is a sum of $S_i$ exponential random variables, each having rate $\sum_{j=1}^{|\mathcal{T}_i^c(t_{n,i})|} \mathcal{D}_{i,j}$.
It follows that 
\begin{align*}
&P\left( E_{n,i} \cap \{ |J_i^*|=0\} \cap  \{ s_i^* <t_{n,i} -\epsilon \} \right) \\
&\leq  P\left( E_{n,i} \cap \{ |J_i^*|=0\} \cap  \{ t_{n,i} -\epsilon > s_i^* \geq \sigma_{S_n} - \sigma_{S_i} -\epsilon/2 \} \right) \\
&\hspace{5mm} + P\left( E_{n,i} \cap \{ |J_i^*|=0\} \cap  \left\{ s_i^* < t_{n,i}, \, \sum_{j=2}^{|J_i|} (\sigma_{S_{\kappa_i(j)}} - \sigma_{S_{\kappa_i(j)-1}}) > \epsilon/2 \right\} \right) \\
&\leq P\left( t_{n,i} -\sigma_{S_n} + \sigma_{S_i} > \epsilon/2 \right) + P\left( E_{n,i} \cap  \{  \mathcal{E}_i > \epsilon/2 \} \right).
\end{align*}
By Markov's inequality we have that
$$P\left( E_{n,i} \cap  \{  \mathcal{E}_i > \epsilon/2 \} \right) \leq P( \text{Erlang}(a_i, S_i) > \epsilon/2 ) \leq  1 \wedge E\left[ \frac{2 a_i}{\epsilon S_i} \right] \leq 1 \wedge \frac{2a_i}{\epsilon i}.$$

Finally, let $I_n = \lceil n U \rceil$, where $U$ is a uniform $[0,1]$ independent of everything else. Then, we conclude that
\begin{align*}
&\frac{1}{n} \sum_{i=1}^n P\left( E_{n,i} \cap \left\{ |J_i^*| = 0 \right\} \cap \left\{  |\mathcal{T}_i^c(t_{n,i})| \neq |\mathcal{G}_i^{(n)}| \right\} \right) \\
&\leq  E \left[ 1 - e^{-C_f \epsilon \Lambda_{I_n}(t_{n,I_n}) }   \right] + E\left[ \frac{2a_{I_n}}{\epsilon I_n} \wedge 1\right] + P\left(  t_{n,I_n} -\sigma_{S_n} + \sigma_{S_{I_n}} > \epsilon/2 \right) \\
&\qquad + \frac{1}{n} \sum_{i=1}^n P\left( |\mathcal{T}_i^c(t_{n,i})| \neq |\mathcal{T}_i^c(\sigma_{S_n} - \sigma_{S_i})| \right)\\
&\to E \left[ 1 - e^{-C_f \epsilon \Lambda(\chi) }   \right] , \qquad n \to \infty,
\end{align*}
by Lemma~\ref{L.TimeAdjustment}(i) and the same arguments used in its proof. Taking $\epsilon \downarrow 0$ completes the proof of (i).

To prove (ii), for any $\ui = (i_1, \ldots, i_m) \subseteq [n]$, note that 
\begin{align*}
    &P \left( \bigcap_{\ell=1}^{m} E_{n,i_{\ell};\ui} \cap \bigcap_{\ell=1}^{m} \{|J_{i_{\ell};\ui}^*| = 0 \} \cap \bigcup_{\ell=1}^{m} \{ |\cT_{i_{\ell};\ui}^{c}(t_{n,i_{\ell}})| \neq |\cG_{i_{\ell}}^{(n)}| \}\right)\\
    &\le \sum_{\ell = 1 }^{m} P \left( \bigcap_{j=1}^{\ell} E_{n,i_{j};\ui} \cap \bigcap_{j=1}^{\ell} \{|J_{i_{j};\ui}^*| = 0 \} \cap  \{ |\cT_{i_{\ell};\ui}^{c}(t_{n,i_{\ell}})| \neq |\cG_{i_{\ell}}^{(n)}| \}\right).
\end{align*}
Then note that (ii) follows from the same arguments as (i) upon observing that, on the event $\bigcap_{j=1}^{\ell} E_{n,i_{j};\ui} \cap \bigcap_{j=1}^{\ell} \{|J_{i_{j};\ui}^*| = 0 \}$, $s^*_{i_\ell; \ui} - \sigma_{S_n} - \sigma_{S_{i_\ell}}$ is stochastically dominated by an $ \text{Erlang}(\ell a_{i_\ell}, S_{i_\ell})$ random variable.

\end{proof} 

We are now ready to prove Theorem~\ref{T.Main}.

\begin{proof}[Proof of Theorem~\ref{T.Main}(i)]
For any $i \in [n]$, let $a_i = i^{1/2-\delta}$ for some $0 < \delta < 1/2$ and $E_{n,i}$ be the event defined in Lemma~\ref{L.LargeTree}. 
To start, define the event 
$$F_{n,i} = \left\{ \mathcal{G}_i^{(n)} \simeq \mathcal{T}_i^c(t_{n,i}), \, \bigcap_{\mathbf{j} \in \mathcal{T}_i^c(t_{n,i})} \{ D_{\theta_i(\mathbf{j})}^+ = \mathcal{D}_{\mathbf{j}} \} \right\},$$
where $\theta_i$ is the bijection defining $ \mathcal{G}_i^{(n)} \simeq \mathcal{T}_i^c(t_{n,i})$. Let $U$ be uniformly distributed in $[0,1]$, independent of everything else, and set $I_n = \lceil nU \rceil$. We will start by showing that $P(F_{n,I_n}) \rightarrow 1$ as $n \to \infty$. Observe that, by virtue of the coupling construction in Section \ref{SS.Coupling}, on the event $\left\{|J_i^*| = 0 \right\} \cap \left\{ |\mathcal{T}_i^c(t_{n,i})| = |\mathcal{G}_i^{(n)}|\right\}$, one can obtain a bijection $\theta_i$ by requiring $\theta_i(\mathbf{j}) := \kappa_i(j)$, where $\kappa_i(j)$ is the enumeration of the vertex in $\mathcal{G}^{(n)}_i$ corresponding to $\mathbf{j} \in \mathcal{T}_i^c(t_{n,i})$ in the exploration of $\mathcal{G}^{(n)}_i$.
From this and Remark~\ref{R.Miscoupling}, note that
\begin{align*}
P\left( F_{n, i} \right) &\geq P\left( \left\{ |J_i^*| = 0 \right\} \cap \left\{ |\mathcal{T}_i^c(t_{n,i})| = |\mathcal{G}_i^{(n)}|\right\} \cap E_{n,i} \right) \\
&\geq 1 - P\left( E_{n,i}^c \right)  - P\left( E_{n,i} \cap \left\{ |J_i^*| \ge 1 \right\} \right) \\
&\hspace{5mm} - P\left( E_{n,i} \cap \left\{ |J_i^*| = 0 \right\} \cap \left\{  |\mathcal{T}_i^c(t_{n,i})| \neq |\mathcal{G}_i^{(n)}| \right\} \right). 
\end{align*}
The limit $P(F_{n,I_n}) \rightarrow 1$ as $n \to \infty$ will follow once we show that
\begin{align*} 
\Delta_n &:= \frac{1}{n} \sum_{i=1}^n \left(\phantom{\sum \mkern-24mu} P(E_{n,i}^c) + P(E_{n,i} \cap \{ |J_i^*| \ge 1\}) \right. \\
&\hspace{5mm}  \left.+ P\left(E_{n,i} \cap \{ |J_i^*| = 0\} \cap \left\{ |\mathcal{T}_i^c(t_{n,i})| \neq |\mathcal{G}_i^{(n)}| \right\}\right) \right) \to 0 
\end{align*}
as $n \to \infty$. Note that by part (i) of Lemmas~\ref{L.TimeAdjustment}, \ref{L.LargeTree}, \ref{L.Miscoupling}, \ref{L.Mistiming} we have that for any $\epsilon > 0$, 
\begin{align*}
\lim_{n \to \infty} \Delta_n &\leq \lim_{n \to \infty} \left\{E\left[ 1( I_n < n^{c_\vartheta}) + 1(I_n = 1) +  \frac{2C_f \mu I_n^{-2\delta}}{f_*} \right] \right\} \\
&\leq \lim_{n \to \infty} P(U < n^{c_\vartheta -1}) = 0.
\end{align*}
This in turn implies that
\begin{equation} \label{eq:DiscreteTime}
P( F_{n, I_n})  \to 1, \qquad n \to \infty.
\end{equation}
Finally, note that, recalling $\chi = -(1/\lambda)\log U$ and setting and $\theta(\mathbf{j}) := \theta_{I_n}(\mathbf{j})$,
\begin{align*}
&\left| P\left( \mathcal{G}_{I_n}^{(n)} \simeq \mathcal{T}_{I_n}^c(\chi), \, \bigcap_{\mathbf{j} \in \mathcal{T}_{I_n}^c(\chi)} \{ D_{\theta(\mathbf{j})}^+ = \mathcal{D}_{\mathbf{j}} \} \right) - 1 \right| \\
&\leq \left| P(F_{n,I_n}) - 1 \right| + P\left( \mathcal{T}_{I_n}^c(\chi) \not\simeq \mathcal{T}_{I_n}^c(t_{n, I_n}) \right).
\end{align*} 
To see that $ P\left( \mathcal{T}_{I_n}^c(\chi) \not\simeq \mathcal{T}_{I_n}^c(t_{n, I_n}) \right) \to 0$ as $n \to \infty$, note that for  $\epsilon_n = n^{-1/2} $,
\begin{align}
&P\left( \mathcal{T}_{I_n}^c(\chi) \not\simeq \mathcal{T}_{I_n}^c(t_{n, I_n}) \right) \notag \\
&= P\left( |\mathcal{T}_{I_n}^c(\chi)| > |\mathcal{T}_{I_n}^c(t_{n, I_n})| \right) \notag  \\
&\leq P(\chi - t_{n, I_n} > \epsilon_n) + P\left(  |\mathcal{T}_{I_n}^c(\chi)| > |\mathcal{T}_{I_n}^c(\chi-\epsilon_n)| \} \right) \notag \\
&\leq P\left( \frac{\lceil nU \rceil}{nU} > e^{\lambda \epsilon_n} \right) +  P\left( \text{Exp}\left( C_f \left( |\mathcal{T}_{I_n}^c(\chi-\epsilon_n)| + \sum_{j \in \mathcal{T}_{I_n}^c(\chi-\epsilon_n)} \mathcal{D}_j \right)  \right) \leq \epsilon_n \right)  \notag \\
&\leq \frac{1}{n( e^{\lambda \epsilon_n} -1)}    + E\left[ 1 - e^{-C_f \epsilon_n \left( |\mathcal{T}_{I_n}^c(\chi)| + \sum_{j \in \mathcal{T}_{I_n}^c(\chi)} \mathcal{D}_j \right)}  \right] \to 0, \label{eq:ContinuousTime}
\end{align}
as $n \to \infty$. Noting that $\mathcal{T}_{I_n}^c(\chi)$ has the same law as $\mathcal{T}^c(\chi)$ (with $\mathcal{T}^c(\cdot)$ independent of $\chi$) completes the proof of Theorem~\ref{T.Main}(i).
\end{proof}

We now give the proof of part (ii) of Theorem~\ref{T.Main}. 

\begin{proof}[Proof of Theorem~\ref{T.Main}(ii)]
Recall the events $F_{n,i}$  defined in the proof of part (i) of Theorem~\ref{T.Main}. Let $U_1, \ldots, U_m$ be independent random variables, uniformly distributed in $[0,1]$, and set $I_{n,\ell} = \ceil{nU_{\ell}}$. We first show that $P(\bigcap_{\ell=1}^{m} F_{n,I_{n,\ell}}) \to 1$ as $n \rightarrow \infty$. Note that for any $\ui = (i_1, \ldots, i_m) \subseteq [n]$, 
\begin{align*}
    P\left(\bigcap_{\ell=1}^{m}F_{n,i_{\ell}}\right) &\geq P\left(\bigcap_{\ell=1}^{m} E_{n,i_{\ell};\ui} \cap \bigcap_{\ell=1}^{m} \{|J_{i_{\ell};\ui}^* | = 0\} \cap \bigcap_{\ell=1}^{m} \left\{ |\cT_{i_{\ell};\ui}^c(t_{n,i_{\ell}})| = |\cG_{i_{\ell}}^{(n)}| \right\}\right)\\
    &\geq 1 - \sum_{\ell =1}^{m}P\left(E_{n,i_{\ell};\ui}^c\right) - P\left(\bigcap_{\ell=1}^{m} E_{n,i_{\ell};\ui} \cap \bigcup_{\ell=1}^m \{|J_{i_{\ell};\ui}^*| \geq 1\}\right)\\ & - P\left(\bigcap_{\ell=1}^{m} E_{n,i_{\ell};\ui} \cap \bigcap_{\ell=1}^{m} \{|J_{i_{\ell};\ui}^* | = 0\} \cap \bigcup_{\ell=1}^m\left\{ |\cT_{i_{\ell};\ui}^c(t_{n,i_{\ell}})| \neq |\cG_{i_{\ell}}^{(n)}| \right\}\right). 
\end{align*}
So it suffices to prove \begin{align*}
&\frac{1}{n^m} \sum_{\ui \subseteq [n] : |\ui|=m} \left[ \sum_{\ell =1}^{m}P(E_{n,i_{\ell};\ui}^c) + P\left(\bigcap_{\ell=1}^{m} E_{n,i_{\ell};\ui} \cap \bigcup_{\ell=1}^m \{|J_{i_{\ell};\ui}^*| \geq 1\}\right) \right.\\ &+ \left. P\left(\bigcap_{\ell=1}^{m} E_{n,i_{\ell};\ui} \cap \bigcap_{\ell=1}^{m} \{|J_{i_{\ell};\ui}^* | = 0\} \cap \bigcup_{\ell=1}^m\left\{ |\cT_{i_{\ell};\ui}^c(t_{n,i_{\ell}})| \neq |\cG_{i_{\ell}}^{(n)}| \right\}\right)\right] \to 0
\end{align*}
as $n \to \infty$. But this follows directly once we similarly apply part (ii) of Lemmas~\ref{L.TimeAdjustment}, \ref{L.Miscoupling}, \ref{L.Mistiming} and Lemma \ref{L.LargeTree}, Remark~\ref{L.LargeTreeR}, Remark~\ref{unizero}. Finally, we have \begin{align*}
\left| \cP_{n,m}\left( \bigcap_{\ell=1}^{m} C_{I_{n,\ell}} \right) - 1 \right| &\leq \left|P\left(\bigcap_{\ell=1}^{m}F_{n,I_{n,\ell}}\right) - 1 \right| + \sum_{\ell=1}^{m}P(\cT_{I_{n,\ell}}^c(\chi_\ell) \not \simeq \cT_{I_{n,\ell}}^c(t_{n,I_{n,\ell}})) \\
&\leq \left|P\left(\bigcap_{\ell=1}^{m}F_{n,I_{n,\ell}}\right) - 1 \right| + m P(\cT_{I_n}^c(\chi) \not \simeq \cT_{I_n}^c(t_{n,I_n})). 
\end{align*}
We already showed that both terms converge to $0$ as $n \to \infty$, which completes the proof.
\end{proof}

\begin{proof}[Proof of Corollary \ref{wconp}]
Fix a finite tree $T$ and a deterministic sequence of marks $\{d_{\jb} : \jb \in \cU \}$. For $n \in \mathbb{N}$ and $\{ I_{n,k}: k= 1, 2\}$ i.i.d.~sampled uniformly from $V_n$, independently of anything else, recall the coupling $\cP_{n,2}$ of $\left(\cG^{(n)}_{I_{n,1}}, \cG^{(n)}_{I_{n,2}}\right)$ with $(\mathcal{T}^c_1(\chi_1), \mathcal{T}^c_2(\chi_2))$, where $\mathcal{T}^c_1(\chi_1)$ and $\mathcal{T}^c_2(\chi_2)$ are i.i.d. copies of $\mathcal{T}^c(\chi)$. Denote the corresponding expectation by $\mathbb{E}_{n,2}$.

To simplify the notation, for $i \in V_n$ and $k=1,2$, define the events $$F_i = \left\{ \cG^{(n)}_i \simeq T, \bigcap_{\jb \in T}  \{D^+_{\theta_i(\jb)} = d_{\jb} \} \right\} \ {\rm and} \ \hat{F}_k = \left\{ \cT^c_k(\chi_k) \simeq T, \bigcap_{\jb \in T}  \{ \mathcal{D}^{(k)}_{\jb} = d_{\jb} \} \right\},$$
where $\theta_i$ is the bijection that defines the isomorphism $\cG^{(n)}_i \simeq T$, and $\mathcal{D}^{(k)}_{\mathbf{j}}$ denotes the mark of the node indexed $\mathbf{j}$ in $\cT^c_k(\chi_k)$. Note that

\begin{align}\label{l2bd}
    &\mathbb{E}_{n,2}\left[ \left(\frac{1}{n} \sum_{i=1}^{n} 1(F_i ) -  \cP_{n,2}(\hat F_1) \right)^2 \right]\notag\\
    & = \mathbb{E}_{n,2}\left[\frac{1}{n^2} \sum_{i,j=1}^{n} 1(F_i )1(F_j)\right] - 2\cP_{n,2}(F_{I_{n,1}}) \cP_{n,2}(\hat F_1) +  \left(\cP_{n,2}(\hat F_1)\right)^2.
\end{align}
By Theorem~\ref{T.Main}(i), with $C_{I_n}$ defined in its statement,
$$
|\cP_{n,2}(F_{I_{n,1}}) - \cP_{n,2}(\hat F_1)| \le \cP_n(C_{I_n}^c) \rightarrow 0
$$
as $n \rightarrow \infty$. Moreover, by Theorem~\ref{T.Main}(ii) (applied for $m=2$), with $C_{I_{n,1}},C_{I_{n,2}}$ defined in its statement,
\begin{align*}
&\left|\mathbb{E}_{n,2}\left[ \frac{1}{n^2} \sum_{i,j=1}^{n} 1(F_i )1(F_j)\right] - \left(\cP_{n,2}(\hat F_1)\right)^2\right|
=\left|\mathbb{E}_{n,2}\left[ \frac{1}{n^2} \sum_{i,j=1}^{n} 1(F_i )1(F_j)\right] - \cP_{n,2}(\hat F_1 \cap \hat F_2)\right|\\
&= \left| \cP_{n,2}(F_{I_{n,1}} \cap F_{I_{n,2}}) - \cP_{n,2}(\hat F_1 \cap \hat F_2)\right| \le \cP_{n,2}(C_{I_{n,1}}^c \cup C_{I_{n,2}}^c) \rightarrow 0
\end{align*}
as $n \rightarrow \infty$. Using the above two observations in \eqref{l2bd}, we conclude
$$
\mathbb{E}_{n,2}\left[ \left( \frac{1}{n} \sum_{i=1}^{n} 1(F_i ) -  P \left( \cT^c(\chi) \simeq T, \, \bigcap_{\jb \in T} \{\mathcal{D}_{\jb} = d_{\jb} \} \right) \right)^2 \right] \rightarrow 0
$$
as $n \rightarrow \infty$. The result follows from this.
\end{proof}

\begin{proof}[Proof of Proposition \ref{lincase}]
We write $g(x) \sim f(x)$ as $x \to \infty$ to denote $\lim_{x \to \infty} g(x)/f(x) = 1$. 

(1) \textit{Preferential attachment:} The expression for $P(\InDeg_\emptyset = x)$ follows from \cite[Corollary 1.4]{garavaglia2018trees}. To get the exponent for a regularly varying out-degree distribution, observe that, by Stirling's formula, 
\begin{align*}
\frac{\Gamma(x + d(\beta + 1))}{\Gamma(x + d(\beta + 1) + 3 + \beta)} &= (x + d(\beta + 1))^{-3-\beta} (1 + O((x + d(\beta + 1))^{-1})), \\
l(d) := \frac{(2+\beta) \Gamma(2+\beta + d(\beta+1)) }{d^{2+\beta} \Gamma(d(\beta + 1))} &=  (2+\beta) (\beta+1)^{2+\beta} (1 + O(d^{-1})).
\end{align*}
Using the expression for $P(\InDeg_\emptyset = x)$  in \cite[Corollary 1.4]{garavaglia2018trees} and the first estimate above, we can write
$$
P(\InDeg_\emptyset = x) = (1 + O(x^{-1}))\sum_{d=1}^{\infty} d^{2+ \beta - \gamma}(x+ (\beta + 1)d)^{-3-\beta} \tilde L(d),
$$
where $\tilde L(d) := l(d) L(d) \sim l(\infty) L(d)$ as $d \to \infty$, where $l(\infty) := \lim_{d \to \infty} l(d)$.

Fix $\epsilon \in (0,1)$ and define $b(x) = \lfloor x^{1+\epsilon} \rfloor$. Note that, using \cite[Proposition 1.5.10]{bingham1989regular},
\begin{align*}
P(\InDeg_\emptyset = x)  &= O\left( \sum_{d=b(x)+1}^\infty \tilde L(d) d^{-1-\gamma} \right) \\
&\hspace{5mm} + (1+ O(x^{-1})) \sum_{d=1}^{b(x)} d^{2+ \beta - \gamma}(x+ (\beta + 1)d)^{-3-\beta} \tilde L(d) \\
&= O\left( b(x)^{-\gamma} \tilde L(b(x)) \right)  + (1+ O(x^{-1})) K_\gamma P(XY > x) 
\end{align*}
where $P(Y = d) = d^{-\gamma-1} \tilde L(d)/K_\gamma$ for $d \in \mathbb{N}$ and $K_\gamma = \sum_{d=1}^\infty d^{-\gamma-1} \tilde L(d)$ and $P(X > x)= (1 + \beta + x)^{-3-\beta}$ for $x > -\beta$ is a Type II Pareto random variable, independent of $Y$. The assertion about regular variation of $x \mapsto P(XY>x)$ follows from \cite[Lemma 4.1]{jessen2006regularly}. By Breiman's theorem (see \cite[Lemma 4.2]{jessen2006regularly}), $$P(XY > x) \sim E[Y^{3+\beta}] P(X > x) \sim E[Y^{3+\beta}] x^{-3-\beta}$$ if $\gamma > 3+\beta$, and $$P(XY > x) \sim E[(X^+)^\gamma] P(Y > x) \sim E[(X^+)^\gamma] K_\gamma^{-1} \tilde L(x) x^{-\gamma}$$ if $2 \le \gamma < 3+\beta$, as $x \to \infty$. Since 
\begin{align*}
 b(x)^{-\gamma} \tilde L(b(x)) &= o\left( x^{- \gamma } \right) = o\left( P(XY > x) \right)
\end{align*}
as $x \to \infty$, we obtain that
\begin{align*}
P(\InDeg_\emptyset = x) &= (1 + o(1))  K_\gamma P(XY > x)
\end{align*}
as $x \to \infty$. 

(2) \textit{Uniform attachment:} The expression for $P(\InDeg_\emptyset = x)$ follows from \cite[Corollary~1.6]{garavaglia2018trees}. To get the lower bound on the in-degree distribution for regularly varying out-degree distribution, observe that
\begin{align*}
P(\InDeg_\emptyset = x) &= \sum_{d=1}^\infty d^{-\gamma-1} L(d) \left(1 + \frac{1}{d} \right)^{-x-1} = E[1/\mathcal{D}] E\left[ \left(1 + \frac{1}{Y'} \right)^{-x-1} \right]  \\
&= E[1/\mathcal{D}]  E\left[ e^{-(x+1) \log(1 + 1/Y')} \right] =  E[1/\mathcal{D}] P\left( \frac{W}{\log(1+1/Y')} > x+1 \right)
\end{align*}
where $W$ is an exponential random variable with rate one, independent of $Y'$, and $P(Y' = d) = d^{-\gamma-1} L(d)/E[1/\mathcal{D}]$ for $d \in \mathbb{N}$. 

Now note that the random variable $V = 1/\log(1+1/Y')$ satisfies, as $x \to \infty$,
\begin{align*}
P(V > x) &= P\left( Y' >  1/(e^{1/x} -1) \right)  = P( Y' > x (1+o(1))) \\
&= (1+o(1)) (E[1/\mathcal{D}])^{-1}  x^{-\gamma} L(x) ,
\end{align*}
and is, therefore, regularly varying with tail index $\gamma$. Breiman's theorem gives now
\begin{align*}
E[1/\mathcal{D}] P(W V > x+1) &= (1+o(1)) E[ W^\gamma] E[1/\mathcal{D}]  P(V > x+1) \\
&= (1+o(1)) E[ W^\gamma]   x^{-\gamma} L(x) 
\end{align*}
as $x \to \infty$. 
\end{proof}

\begin{proof}[Proof of Proposition~\ref{P.logasymptotics}.]
For any $x>0$, $m \in \mathbb{N}$, we will write $y = y(x,m) = x/m$.

{\bf Upper bound.} Fix $m \in \mathbb{N}$. Choose $\kappa = \kappa(y) = c \lambda G_2(\lceil y \rceil)$ with $c > 1$. For any $n \in \mathbb{N}$, define $S(n) := \inf \{t \geq 0 : \xi_f(t) > n \}$ and use the union bound to obtain for $x>0$,
\begin{align*}
P\left( \frac{1}{m} \sum_{i=1}^m \xi_f^{(i)}(\chi) > y \right) &\leq P\left( \frac{1}{m} \sum_{i=1}^m \xi_f^{(i)}(\chi) > y, \, \chi \leq G_1(\lceil y \rceil) - \kappa  \right) + P\left( \chi > G_1(\lceil y \rceil) - \kappa \right) \\
&\leq   m P\left(  \xi_f(\chi) > y , \, \chi \leq G_1(\lceil y \rceil) - \kappa  \right) + e^{-\lambda (G_1(\lceil y \rceil) - \kappa)} \\
&= m P\left( S(\lfloor y \rfloor+1) \leq \chi, \, \chi \leq G_1(\lceil y \rceil) - \kappa  \right) + e^{-\lambda (G_1(\lceil y \rceil) - \kappa)} \\
&\leq m P\left( S(\lceil y \rceil) \leq \chi, \, \chi \leq G_1(\lceil y \rceil) - \kappa  \right) + e^{-\lambda (G_1(\lceil y \rceil) - \kappa)}.
\end{align*}

Next, use Chernoff's bound, the m.g.f.~of $S(\lceil y \rceil)$, and the inequality $\log(1+x) \geq x - x^2/2$ for $x \geq 0$, to obtain that, 
\begin{align*}
&P\left(   S(\lceil y \rceil) \le \chi, \,  \chi \leq G_1(\lceil y \rceil) - \kappa  \right) \\
&= E\left[ P\left( \left. S(\lceil y \rceil) \leq \chi \right| \chi \right) 1(\chi \leq G_1(\lceil y \rceil) - \kappa ) \right] \\
&= E\left[ P\left( \left. G_1(\lceil y \rceil ) -S(\lceil y \rceil) \geq G_1(\lceil y \rceil) -\chi \right| \chi \right) 1(\chi \leq G_1(\lceil y \rceil) - \kappa ) \right] \\
&\leq E\left[ \inf_{\theta > 0} e^{-\theta (G_1(\lceil y \rceil) - \chi)} E\left[ e^{\theta (G_1(\lceil y \rceil) - S(\lceil y \rceil)} \right] 1(\chi \leq G_1(\lceil y \rceil) - \kappa )\right]  \\
&= E\left[ \inf_{\theta > 0} e^{-\theta (G_1(\lceil y \rceil) - \chi)} e^{\sum_{i=1}^{\lceil y \rceil} \left( \frac{\theta}{f(i)} - \log(1 + \theta/f(i)) \right)} 1(\chi \leq G_1(\lceil y \rceil) - \kappa )\right]\\
&\leq E\left[ \inf_{\theta > 0} e^{-\theta (G_1(\lceil y \rceil) - \chi) + \frac{\theta^2}{2} G_2(\lceil y \rceil)} 1(\chi \leq G_1(\lceil y \rceil) - \kappa )\right]\\
&= E\left[ e^{- \frac{(G_1(\lceil y \rceil) - \chi)^2}{2 G_2(\lceil y \rceil)} } 1(\chi \leq G_1(\lceil y \rceil) - \kappa )\right].
\end{align*}
Now, recalling $\chi$ has an exponential distribution with rate $\lambda$, the above bound leads to
\begin{align*}
&P\left(   S(\lceil y \rceil) \le \chi, \,  \chi \leq G_1(\lceil y \rceil) - \kappa  \right)\\
&\le \int_{0}^{G_1(\lceil y \rceil)-\kappa} \lambda e^{-\lambda t - \frac{(G_1(\lceil y \rceil) - t)^2}{2 G_2(\lceil y \rceil)} } dt \\
&= \lambda G_2(\lceil y \rceil)^{1/2}e^{-\lambda G_1(\lceil y \rceil)}  \int_{-G_1(\lceil y \rceil)/G_2(\lceil y \rceil)^{1/2}}^{-\kappa/G_2(\lceil y \rceil)^{1/2} }  e^{-\lambda  G_2(\lceil y \rceil)^{1/2} z  - \frac{z^2}{2 } } dz \\
&= \lambda G_2(\lceil y \rceil)^{1/2} e^{-\lambda G_1(\lceil y \rceil) + \lambda^2 G_2(\lceil y \rceil)/2 }  \int_{-G_1(\lceil y \rceil)/G_2(\lceil y \rceil)^{1/2}}^{-\kappa/G_2(\lceil y \rceil )^{1/2}} e^{-(z + \lambda G_2(\lceil y \rceil)^{1/2})^2/2} dz \\
&\leq \sqrt{2\pi} \lambda G_2(\lceil y \rceil)^{1/2} e^{-\lambda G_1(\lceil y \rceil) + \lambda^2 G_2(\lceil y \rceil)/2 } \Phi\left( - \kappa/G_2(\lceil y \rceil)^{1/2} + \lambda G_2(\lceil y \rceil)^{1/2} \right) \\
&\leq e^{-\lambda G_1(\lceil y \rceil) + \lambda^2 G_2(\lceil y \rceil)/2 }  e^{-\frac{1}{2} \left( \kappa/G_2(\lceil y \rceil)^{1/2} - \lambda G_2(\lceil y \rceil)^{1/2} \right)^2 } \frac{\lambda G_2(\lceil y \rceil)^{1/2} }{\kappa/G_2(\lceil y \rceil)^{1/2} - \lambda G_2(\lceil y \rceil)^{1/2}} \\
&= \frac{\lambda G_2(\lceil y \rceil)}{\kappa - \lambda G_2(\lceil y \rceil)} e^{-\lambda G_1(\lceil y \rceil)  - \frac{\kappa^2}{2 G_2(\lceil y \rceil)}  + \lambda \kappa  } .
\end{align*}
It follows from the union bound that
\begin{align*}
P\left( \frac{1}{m} \sum_{i=1}^m \xi_f^{(i)}(\chi) > y \right) &\leq P\left( \frac{1}{m} \sum_{i=1}^m \xi_f^{(i)}(\chi) > y, \, \chi \leq G_1(\lceil y \rceil) - \kappa  \right) + P\left( \chi > G_1(\lceil y \rceil) - \kappa \right) \\
&\leq m P\left( \xi_f(\chi) > y, \, \chi \leq G_1(\lceil y \rceil) -\kappa \right) + P( \chi > G_1(\lceil y \rceil) - \kappa ) \\
&= m P\left( S(\lceil y \rceil+1) \leq \chi , \, \chi \leq G_1(\lceil y \rceil) - \kappa \right) + e^{-\lambda (G_1(\lceil y \rceil) - \kappa)} \\
&\leq \frac{m\lambda G_2(\lceil y \rceil)}{\kappa - \lambda G_2(\lceil y \rceil)} e^{-\lambda G_1(\lceil y \rceil)  - \frac{\kappa^2}{2 G_2(\lceil y \rceil)}  + \lambda \kappa  } +  e^{-\lambda (G_1(\lceil y \rceil) - \kappa)} \\
&= \frac{m c}{c-1} e^{-\lambda G_1(\lceil y \rceil)  + \left( 1 - \frac{c}{2} \right) \lambda^2 c G_2(\lceil y \rceil)} + e^{-\lambda G_1(\lceil y \rceil) + \lambda^2 c G_2(\lceil y \rceil)} .
\end{align*}
To get the final upper bound, note that
\begin{align*}
    P\left(\sum_{i=1}^{\mathcal{D}} \xi_f^{(i)}(\chi) > x \right) &= \sum_{m=1}^{\infty} P\left(\left.\frac{1}{m}\sum_{i=1}^{m} \xi_f^{(i)}(\chi) > y \right| \cD =m \right)P(\cD=m)\\
    &\leq E \left[ \frac{c \cD}{c-1}e^{-\lambda G_1(\lceil \frac{x}{\cD} \rceil)  + \left( 1 - \frac{c}{2} \right) \lambda^2 c G_2(\lceil \frac{x}{\cD} \rceil)} + e^{-\lambda G_1(\lceil \frac{x}{\cD} \rceil) + \lambda^2 c G_2(\lceil \frac{x}{\cD} \rceil)}\right].
\end{align*}
The desired bound follows on taking $c=2$.

{\bf Lower bound.} For the lower bound, first fix any $m \in \mathbb{N}$. Set $\kappa = \kappa(y) = Mf_* G_2(\lceil y \rceil)$ where we recall $f_* = \inf_{i \ge 1} f(i)$ and $M = M(m) \ge 2$ will be appropriately chosen later. Using the memoryless property of exponential random variables, for any $x>0$, recalling $y=x/m$,
\begin{align*}
P\left( \frac{1}{m} \sum_{i=1}^m \xi_f^{(i)}(\chi) > y \right) &\geq P\left( \frac{1}{m} \sum_{i=1}^m \xi_f^{(i)}(\chi) > y, \, \chi > G_1(\lceil y \rceil)  + \kappa  \right)  \\
&= e^{-\lambda (G_1( \lceil y \rceil)  + \kappa)} P\left( \frac{1}{m} \sum_{i=1}^m \xi_f^{(i)}(\chi + G_1( \lceil y \rceil )  + \kappa) > y \right) \\
&\geq e^{-\lambda (G_1( \lceil y \rceil)  + \kappa)} P\left( \frac{1}{m} \sum_{i=1}^m \xi_f^{(i)}(G_1( \lceil y \rceil )  + \kappa) > y \right)\\
& \ge e^{-\lambda (G_1( \lceil y \rceil)  + \kappa)}  \left(1 - mP\left( \xi_f(G_1( \lceil y \rceil )  + \kappa) \le y \right)\right).
\end{align*}
Observe that,
\begin{align*}
P\left( \xi_f(G_1( \lceil y \rceil )  + \kappa) \le y \right) &\le P\left(S(\lceil y \rceil) \ge G_1(\lceil y \rceil )  + \kappa\right)\\
& \le e^{-\frac{f_*}{2}\left(G_1( \lceil y \rceil )  + \kappa\right) - \sum_{i=1}^{\lceil y \rceil}\log\left(1-\frac{f_*}{2f(i)}\right)}
 = e^{-\frac{f_*}{2}\left(G_1( \lceil y \rceil)  + \kappa\right) + \sum_{k=1}^{\infty}\frac{f_*^k}{k2^k}G_k(\lceil y \rceil)}\\
 & \le e^{-\frac{f_*}{2} \kappa + f_*^2G_2(\lceil y \rceil)\sum_{k=2}^{\infty}\frac{1}{k 2^k}}
 \le e^{-\frac{f_*}{2} \kappa + \frac{f_*^2}{2}G_2(\lceil y \rceil)} \le e^{-\frac{Mf_*^2}{4}G_2(\lceil y \rceil)}.
\end{align*}
Combining the above estimates, we obtain for any $m \in \mathbb{N}$,
\begin{align*}
 P\left(\sum_{i=1}^{m} \xi_f^{(i)}(\chi) > x \right) \ge e^{-\lambda (G_1( \lceil \frac{x}{m}\rceil )  + M f_* G_2(\lceil \frac{x}{m}\rceil))}\left(1 - m e^{-\frac{Mf_*^2}{4}G_2(\lceil \frac{x}{m}\rceil)}\right).
\end{align*}
Setting $M= \frac{4f(1)^2}{f_*^2}\log (2m) \ge 2$, note that $1 - m e^{-\frac{Mf_*^2}{4}G_2(\lceil \frac{x}{m}\rceil)} \ge 1 - m e^{-\frac{Mf_*^2}{4f(1)^2}}  = \frac{1}{2}$ for every $m \in \mathbb{N}$ and $x>0$.
Hence, we obtain the following lower bound
\begin{align*}
 P\left(\sum_{i=1}^{\mathcal{D}} \xi_f^{(i)}(\chi) > x \right) \ge \frac{1}{2}E \left[e^{-\lambda (G_1( \lceil \frac{x}{\mathcal{D}}\rceil )  + \frac{4f(1)^2}{f_*}\log (2\mathcal{D}) \, G_2(\lceil \frac{x}{\mathcal{D}}\rceil))}\right].
\end{align*}
\end{proof}

The final proof in the paper corresponds to the weak convergence of $\{\xi_f(t): t \geq 0\}$. 

\begin{proof}[Proof of Proposition \ref{L.WeakConv}]
To start, note that
\begin{align*}
\lim_{t \to \infty} \frac{V(t)}{G_1^{-1}(t)} &= \lim_{x \to \infty} \left( \frac{f(x) G_2(\infty)^{1/2}}{x} 1(G_2(\infty) < \infty) + \frac{f(x) G_2(x)^{1/2}}{x} 1(G_2(\infty) = \infty) \right) = 0,
\end{align*}
so $V(t) = o(G_1^{-1}(t))$ as $t \to \infty$. 

Next, note that if $\{ \mathcal{E}_i: i \geq 1\}$ are i.i.d.~exponential random variables with mean one, then, for any $z \geq -G_1^{-1}(t)/V(t)$, 
\begin{align*}
P\left( \frac{\xi_f(t) - G_1^{-1}(t)}{V(t)} > z \right) &= P\left( \xi_f(t) > G_1^{-1}(t) + z V(t) \right) \\
&= P\left( \sum_{i=1}^{ \lfloor G_1^{-1}(t) + z V(t) \rfloor +1} \frac{\mathcal{E}_i}{f(i)} \leq t \right) \\
&= P\left( \sum_{i=1}^{ \lfloor G_1^{-1}(t) + z V(t) \rfloor +1} \frac{1-\mathcal{E}_i}{f(i)} \geq G_1(\lfloor G_1^{-1}(t) + z V(t) \rfloor+1)  -t \right).
\end{align*}
Now let  $\tilde Z(n) =  G_2(n)^{-1/2} \sum_{i=1}^n (1-\mathcal{E}_i)/f(i)$. Note that for any $|\theta|<1$, 
\begin{align*}
E\left[ e^{\theta \tilde Z(n)} \right] &= e^{\theta G_2(n)^{-1/2} G_1(n) }  \prod_{i=1}^{n} E\left[ e^{-\theta G_2(n)^{-1/2} \mathcal{E}_i/f(i)} \right] = e^{\theta u(n)^{-1/2} G_1(n) }  \prod_{i=1}^{n} \frac{1}{1 + \theta G_2(n)^{-1/2} /f(i)} \\
&= e^{\theta G_2(n)^{-1/2} G_1(n)  - \sum_{i=1}^{n} \log (1 + \theta G_2(n)^{-1/2}/f(i)) }  \\
&= e^{\theta G_2(n)^{-1/2} G_1(n)  - \sum_{i=1}^{n} \sum_{k=1}^\infty \frac{(-1)^{k-1} \theta^k}{k} \cdot \frac{1}{G_2(n)^{k/2} f(i)^k}  } \\
&= e^{\theta G_2(n)^{-1/2} G_1(n)  - \sum_{k=1}^\infty \frac{(-1)^{k-1} \theta^k}{k} \cdot \frac{G_k(n)}{G_2(n)^{k/2} }} \\
&=  e^{ \sum_{k=2}^\infty \frac{(-1)^{k} \theta^k}{k} \cdot \frac{G_k(n)}{G_2(n)^{k/2} }}  \to E[ e^{\theta Z} ],
\end{align*}
as $n \to \infty$, where in the last step we used the observation that if $G_2(\infty) = \infty$, then $G_k(n)/G_2(n)^{k/2} \to 0$ as $n \to \infty$ for all $k \geq 3$ and if $G_2(\infty)< \infty$ then $G_k(n)/G_2(n)^{k/2} \to a_k$. It follows that 
$$P\left( \frac{\xi_f(t) - G_1^{-1}(t)}{V(t)} > z \right)  = P\left( \tilde Z( \lfloor G_1^{-1}(t) + z V(t) \rfloor +1) \geq \frac{G_1(\lfloor G_1^{-1}(t) + z V(t) \rfloor+1)  -t }{G_2( \lfloor G_1^{-1}(t) + z V(t) \rfloor +1)^{1/2}}   \right) ,$$
where $ \tilde Z( \lfloor G_1^{-1}(t) + z V(t) \rfloor +1) \Rightarrow Z$ as $t \to \infty$ for any fixed $z \geq -G_1^{-1}(t)/V(t)$. 
To complete the proof, note that for any fixed $z$, 
\begin{align*}
\frac{G_1(\lfloor G_1^{-1}(t) + z V(t) \rfloor+1)  -t }{G_2( \lfloor G_1^{-1}(t) + z V(t) \rfloor +1)^{1/2}}  &= \frac{z V(t)}{f(G_1^{-1}(t)) G_2(G_1^{-1}(t))^{1/2} } (1+o(1) ) = z (1+ o(1))
\end{align*}
as $t \to \infty$, and therefore, observing that $-G_1^{-1}(t)/V(t) \to -\infty$ as $t \to \infty$, we have for any $z \in \mathbb{R}$,
$$P\left( \tilde Z( \lfloor G_1^{-1}(t) + z V(t) \rfloor +1) \geq \frac{G_1(\lfloor G_1^{-1}(t) + z V(t) \rfloor+1)  -t }{G_2( \lfloor G_1^{-1}(t) + z V(t) \rfloor +1)^{1/2}}   \right)  \to P(Z \geq z), \qquad t \to \infty.$$
\end{proof}

\section*{Acknowledgements}
We thank Shankar Bhamidi for helpful discussions and advice throughout the course of writing this article.
Research supported in part by  the NSF RTG award (DMS-2134107).  Sayan Banerjee and Prabhanka Deka were supported in part by the NSF-CAREER award (DMS-2141621).





\end{document}